\documentclass{aims}
\usepackage{amsmath}
  \usepackage{paralist}
  \usepackage{graphics} 
  \usepackage{epsfig} 
\usepackage{graphicx}  \usepackage{epstopdf}
 \usepackage[colorlinks=true]{hyperref}
\hypersetup{urlcolor=blue, citecolor=red}

\def\currentvolume{X}
 \def\currentissue{X}
  \def\currentyear{200X}
   \def\currentmonth{XX}
    \def\ppages{X--XX}
     \def\DOI{10.3934/xx.xx.xx.xx}

\newtheorem{theorem}{Theorem}[section]
\newtheorem{corollary}{Corollary}

\newtheorem{lemma}[theorem]{Lemma}

\theoremstyle{definition}

\newtheorem{remark}{Remark}

\usepackage{graphicx} 
\usepackage{mathtools}
\usepackage{enumitem}
\usepackage{xcolor}
\usepackage{tikz}
\usetikzlibrary{calc}

\newtheorem{assumption}{Assumption}
\newcommand{\dd}{{~\mathrm{d}}}
\newcommand{\ad}{{\mathrm{ad}}}
\newcommand{\ri}{{\mathrm{ri}}}
\newcommand{\Ker}{{\mathrm{Ker}}}
\newcommand{\Ima}{{\mathrm{Im}}}

\newcommand{\id}{{\mathrm{id}}}

\newcommand{\imp}{{\mathrm{imp}}}
\newcommand{\ft}{{\mathrm{ft}}}
\newcommand{\I}{{\imath}}

\DeclareMathOperator*{\argmin}{argmin}
\DeclareMathOperator*{\curl}{curl}
\DeclareMathOperator*{\dive}{div}

\title[Application examples of minimization based formulations] 
      {Some application examples of minimization based formulations of inverse problems and their regularization}

\author[Kha Van Huynh and Barbara Kaltenbacher]{}

\subjclass{Primary: 65J20; 35R30; Secondary: 35Q60.}
 \keywords{inverse problems for PDEs, regularization, electrical impedance tomography, identification of magnetic permeability, localization of sound sources}

 \email{van.huynh@aau.at}
 \email{barbara.kaltenbacher@aau.at}

\thanks{supported by the Austrian Science Fund FWF under grant P30054}

\thanks{$^*$ Corresponding author: Barbara Kaltenbacher}

\begin{document}
\maketitle

\centerline{\scshape Kha Van Huynh}
\medskip
{\footnotesize
 \centerline{Department of Mathematics,
Alpen-Adria-Universit\"at Klagenfurt}
   \centerline{9020 Klagenfurt, Austria}
} 

\medskip

\centerline{\scshape Barbara Kaltenbacher$^*$ }
\medskip
{\footnotesize
 \centerline{Department of Mathematics,
Alpen-Adria-Universit\"at Klagenfurt}
   \centerline{9020 Klagenfurt, Austria}
}

\bigskip

 \centerline{(Communicated by the associate editor name)}


\begin{abstract}
In this paper we extend a recent idea of formulating and regularizing inverse problems as minimization problems, so without using a forward operator, thus avoiding explicit evaluation of a parameter-to-state map. We do so by rephrasing three application examples in this minimization form, namely (a) electrical impedance tomography with the complete electrode model (b) identification of a nonlinear magnetic permeability from magnetic flux measurements (c) localization of sound sources from microphone array measurements.
To establish convergence of the proposed regularization approach for these problems, we first of all extend the existing theory. In particular, we take advantage of the fact that observations are finite dimensional here, so that inversion of the noisy data can to some extent be done separately, using a right inverse of the observation operator.
This new approach is actually applicable to a wide range of real world problems. 
\end{abstract}

\section{Introduction}
An inverse problem of recovering some quantity $x$ from data $y$ can be expressed as an operator equation $F(x)=y$ with a forward operator $F$.  In practice, the quantity $x$ is usually contained in a mathematical model, e.g., a PDE or an ODE, involving also a state $u$, abstractly written as 
\begin{equation} \label{eq:model}
E(x,u)=0,
\end{equation}
and the data $y$ is collected from observations of the state $u$
\begin{equation} \label{eq:observation}
C(u)=y, 
\end{equation}
where $E$ and $C$ are mappings acting on function spaces
$$E: \widetilde{\mathcal{D}} \times V \to W, \qquad C: V \to Y$$
where $\widetilde{\mathcal{D}} \subseteq X$ and $X,Y,V,W$ are Banach spaces.
In this setting, $F=C \circ S$ is a composite function that concatenates the operator $C$ with the parameter-to-state map $S:\mathcal{D} \to V$ defined by $E(x,S(x))=0, \forall x \in \mathcal{D}$.
However, in order for $S$ to be well-defined, often restrictive assumptions on the parameter need to be made, i.e., the domain $\mathcal{D}$ of $S$ and $F$ will typically only be a strict subset of $\widetilde{\mathcal{D}}$.
Thus our aim is to completely avoid appearance of $S$, as has it already been previously done by means of all-at-once formulations, i.e., by considering \eqref{eq:model} and \eqref{eq:observation} as a system of equations for $x$ and $u$, see, e.g., \cite{KKV14,Kal16}.

An even more general approach to do so is to rewrite (\ref{eq:model}), (\ref{eq:observation}) as an equivalent  minimization problem
\begin{equation} \label{eq:minIP}
(x, u) \in \argmin \{\mathcal{J} (x, u; y): (x, u) \in M_\ad(y)\}
\end{equation}
for some cost function $\mathcal{J}$ and some admissible set $M_\ad(y)$, see \cite{Kal18}. 
Since we are interested in ill-posed problems and, instead of the exact data $y$, only a noisy version $y^{\delta}$ is available, we regularize this by considering
\begin{equation}\label{eq:minIP_noise}
\begin{split}
(x_\alpha^\delta, u_\alpha^\delta) \in \argmin \{ T_\alpha(x, u; y^\delta) = \mathcal{J} (x, u; y^\delta) + \alpha \cdot \mathcal{R}(x, u): \qquad \qquad \\
(x, u)\in M_\ad^\delta(y^\delta) \}.
\end{split}
\end{equation}
The regularization parameter $\alpha \in \mathbb{R}^m_+$ will be chosen according to the noise level $\delta$, the mapping $\mathcal{R}: X \times V \to \overline{\mathbb{R}}^m_+$ corresponds to regularization terms, and the set $M_\ad^\delta (y^\delta) \subset X \times V$ may contain additional constraints that can be used for stabilizing the problem in the sense of Ivanov regularization, see, e.g., \cite{KK18, IVT02, KRR16, LW13, NR14}.

\medskip 

The aim of this paper is to apply this approach to three exemplary practical problems, namely
\begin{itemize}
\item electrical impedance tomography using the complete electrode model;
\item determination of the nonlinear magnetic permeability from measurements of the magnetic flux;
\item localization of sound sources from microphone array measurements.
\end{itemize}

What these and many othe real world problems have in common is -- among others -- finite dimensionality of the observation space.
Therefore, inverting the observation operator (using a right inverse) is stable and to some extent allows to uncouple data inversion from the actual reconstruction process. This approach, which we think is of interest also on its own and applicable to a wide range of other problems, is described in Section \ref{sec:Preliminaries} where more concrete choices of the cost function $\mathcal{J}$ and the admissible set $M_\ad^\delta$ will be given. To do so, we require a somewhat extended version of the abstract background from  \cite[Section 3.1]{Kal18}, which we therefore provide in Section \ref{sec:abstractanalysis}. The main Section \ref{sec:ex} of this paper deals with the three above mentioned application examples in the respective subsections \ref{sec:CEM-EIT}, \ref{sec:MPP}, \ref{sec:SoundSources}.

\section{Convergence analysis}
\subsection{The abstract convergence theory revisited}\label{sec:abstractanalysis}
Like in \cite{Kal18}, we start with a completely general setting, which we slightly extend for our purpose (cf. Remark \ref{rem:adjust_ass:Kal18}). To this end, we first of all very briefly recall and summarize the assumptions and results from \cite[Section 3.1]{Kal18} on existence of minimizers $(x^\delta, u^\delta) := \left( x^\delta_{\alpha (\delta, y^\delta)}, u^\delta_{\alpha (\delta, y^\delta)} \right)$ of (\ref{eq:minIP_noise}) as well as their stability and convergence to a minimizer $(x^\dagger, u^\dagger)$ of (\ref{eq:minIP}).

\begin{assumption} \label{ass:Kal18}
{(\cite[Assumption 3.7]{Kal18})} Let a topology $\mathcal{T}$ and a norm $\|\cdot\|_B$ on $X \times V$ and $\bar \delta >0$ exist such that for the family of noisy data $(y^\delta)_{\delta \in (0,\bar \delta]}$ and any sequence $(y_n)_{n \in \mathbb{N}} \subset Y$ with $y_n \to y^\delta$ in $Y$ and for all
$$ (\tilde x^\delta, \tilde u^\delta) \in \argmin \left\{ \mathcal{J} (x, u; y^\delta) + \alpha (\delta, y^\delta) \cdot \mathcal{R}(x, u): (x, u) \in M_\ad^\delta(y^\delta) \right\} $$
(part of the assumptions below indeed guarantee that this set of minimizers is nonempty)
we have
\begin{enumerate}[label= (\alph*)]
\item $\forall \delta \in (0, \bar \delta], \exists n_0 \in \mathbb{N}, \forall n \ge n_0: (x^\dagger, u^\dagger) \in M_\ad^\delta (y^\delta) \cap M_\ad^\delta (y_n)$;
\item $\forall j \in \{ 1, \dots, m\}: \left( \mathcal{R}_j (x^\dagger, u^\dagger) < \infty \text{\; and \;} \exists \underline{r} \in \mathbb{R}: \mathcal{R}_j \ge \underline{r} \right)$;
\item $\forall c \in \mathbb{R}, \forall \alpha \in \mathbb{R}_+^m$, the sets  $\{ (x, u) \in \bigcup_{\delta \in (0, \bar \delta]} M_\ad^\delta (y^\delta): T_\alpha (x, u; y) \le c \}$ and $\{ (x, u) \in M_\ad^\delta (y^\delta): T_\alpha (x, u; y^\delta) \le c \}$ are $\mathcal{T}$ relatively compact for all $\delta \in (0, \bar \delta]$;
\item $\forall \alpha \in \mathbb{R}_+^m, \forall \delta \in (0, \bar \delta], \exists C_\alpha \ge T_\alpha (x^\dagger, u^\dagger; y^\delta), \exists \tilde{C}_\alpha > 0, \forall (x_n, u_n)_{n \in \mathbb{N}} \subset X \times V, (x_n, u_n) \in M_\ad^\delta (y_n):$
$$ \left( \forall n \in \mathbb{N}: T_\alpha (x_n, u_n; y^\delta) \le C_\alpha \right) \Rightarrow \left( \forall n \in \mathbb{N}: \| (x_n, u_n) \|_B \le \tilde{C}_\alpha \right); $$
\item $\forall \delta \in (0, \bar \delta]: M_\ad^\delta (y^\delta)$ is $\mathcal{T}$ closed;
\item $\forall (\delta_n)_{n \in \mathbb{N}}, (z_n)_{n \in \mathbb{N}} \subset Y, (x_n, u_n)_{n \in \mathbb{N}} \subset X \times V$ with $(x_n, u_n) \in M_\ad^{\delta_n} (z_n)$:
$$\left( \delta_n \to 0, \; z_n \to y, \; (x_n, u_n) \xrightarrow{\mathcal{T}} (x_0, u_0) \right) \Rightarrow \Big( (x_0, u_0) \in M_\ad (y) \Big);$$
\item $\forall j \in \{1,\dots,m\}: \mathcal{R}_j$ is $\mathcal{T}$ lower semicontinuous;
\item $\forall \delta \in (0, \bar \delta], \mathcal{J} (\cdot, \cdot; y^\delta) $ and $\mathcal{J} (\cdot, \cdot; y)$ are $\mathcal{T}$ lower semicontinuous;
\item $\forall \delta \in (0, \bar \delta]: \sup_{(x, u) \in \cup_{m\in\mathbb{N}} M_\ad^\delta (y_m)} \left| \mathcal{J} (x, u; y_n) - \mathcal{J} (x, u; y^\delta) \right| \to 0$ as $n \to \infty$;
\item $\limsup_{\delta \to 0} \sup \left\{ \mathcal{J} (x, u; y) - \mathcal{J} (x, u; y^\delta): (x, u) \in \bigcup_{d \in (0, \bar \delta]} M_\ad^d (y^d) \right\} \le 0$ if $y^\delta \to y$ in $Y$ as $\delta \to 0$;
\item If $y^\delta \to y$ in $Y$ as $\delta \to 0$ then $\forall j \in \{1,\dots,m\}:$\\
$\limsup_{\delta \to 0} \frac{1}{\alpha_j (\delta, y^\delta)} \left( \mathcal{J} (x^\dagger, u^\dagger; y^\delta) - \mathcal{J} (x^\dagger, u^\dagger; y) \right) < \infty$, \\
$\limsup_{\delta \to 0} \frac{1}{\alpha_j (\delta, y^\delta)} \left( \mathcal{J} (x^\dagger, u^\dagger; y) - \mathcal{J} (\tilde x^\delta, \tilde u^\delta; y^\delta) \right) < \infty$, \\
and $\alpha (\delta, y^\delta) \to 0$ as $\delta \to 0$.
\end{enumerate}
\end{assumption}

\begin{remark} \label{rem:adjust_ass:Kal18}
According to a careful check of the proofs of Proposition 3.4 and Theorem 3.6 in \cite{Kal18}, the sets, over which the suprema in Assumption \ref{ass:Kal18}(i),(j) are taken, can be shrunk under conditions
\begin{equation} \label{eq:adjust_ass:Kal18_J(x,u,y)_bounded}
\mathcal{J} (x^\dagger, u^\dagger; y) < \infty
\end{equation}
and
\begin{equation} \label{eq:adjust_ass:Kal18_J(x,u,.)_continuous}
\mathcal{J} (x^\dagger, u^\dagger; \cdot) \text{\; is continuous on a suitable subset of \;} Y.
\end{equation}
Assuming without loss of generality that $\max\{\alpha_j: j=1,\dots,m\} \le 1$, we replace those sets as follows.
\begin{itemize}
\item For Proposition 3.4 in \cite{Kal18}, the minimizers $(x^\delta_{\alpha n}, u^\delta_{\alpha n}) \in \argmin \{ \mathcal{J} (x, u; y_n) + \alpha (\delta, y_n) \cdot \mathcal{R}(x, u): (x,u) \in M_\ad^\delta (y_n)\}$ satisfy
\begin{equation*}
\begin{split}
\mathcal{J} (x^\delta_{\alpha n}, u^\delta_{\alpha n}; y_n) &\le \mathcal{J} (x^\dagger, u^\dagger; y_n) + \alpha \cdot \mathcal{R} (x^\dagger, u^\dagger) - \alpha \cdot \mathcal{R} (x^\delta_{\alpha n}, u^\delta_{\alpha n}) \\
&\le \mathcal{J} (x^\dagger, u^\dagger; y_n) - \mathcal{J} (x^\dagger, u^\dagger; y^\delta) + \mathcal{J} (x^\dagger, u^\dagger; y^\delta) \\
& \qquad + \alpha \cdot \mathcal{R} (x^\dagger, u^\dagger) - \alpha \cdot \mathcal{R} (x^\delta_{\alpha n}, u^\delta_{\alpha n}) \\
& \le 1 + \mathcal{J} (x^\dagger, u^\dagger; y^\delta) + \sum_{j: \mathcal{R}_j (x^\dagger, u^\dagger) \ge 0} \mathcal{R}_j (x^\dagger, u^\dagger) + m \max \{ 0, -\underline{r} \} \\
&=: c (x^\dagger, u^\dagger; y^\delta),
\end{split}
\end{equation*}
where $\mathcal{J} (x^\dagger, u^\dagger; y_n) - \mathcal{J} (x^\dagger, u^\dagger; y^\delta) \le 1$ with $n$ large enough due to the continuity of $\mathcal{J} (x^\dagger, u^\dagger; \cdot)$ at $y^\delta$. So the set $\cup_{m\in\mathbb{N}} M_\ad^\delta (y_m)$ in Assumption \ref{ass:Kal18}(i) can be replaced by
\begin{equation} \label{eq:adjust_ass:Ka18_set(i)}
\bigcup_{m\in\mathbb{N}} \left( M_\ad^\delta (y_m) \cap \{ (x,u): \mathcal{J} (x,u;y_m) \le c (x^\dagger, u^\dagger; y^\delta) \} \right).
\end{equation}
\item Likewise, for Theorem 3.6 in \cite{Kal18}, the minimizers $(\tilde x^\delta, \tilde u^\delta) \in \argmin \{ \mathcal{J} (x, u; y^\delta) + \alpha (\delta, y^\delta) \cdot \mathcal{R}(x, u): (x, u) \in M_\ad^\delta(y^\delta) \}$ satisfy
\begin{equation*}
\begin{split}
\mathcal{J} (\tilde x^\delta, \tilde u^\delta; y^\delta) \le c (x^\dagger, u^\dagger; y),
\end{split}
\end{equation*}
since $\mathcal{J} (x^\dagger, u^\dagger; y^\delta) - \mathcal{J} (x^\dagger, u^\dagger; y) \le 1$ for $\delta$ small enough due to the continuity of $\mathcal{J} (x^\dagger, u^\dagger; \cdot)$ at $y$. So the set $\cup_{d \in (0, \bar \delta]} M_\ad^d (y^d)$ in Assumption \ref{ass:Kal18}(j) can be replaced by
\begin{equation} \label{eq:adjust_ass:Ka18_set(i)}
\bigcup_{d \in (0, \bar \delta]} \left( M_\ad^d (y^d) \cap \{ (x,u): \mathcal{J} (x,u;y^d) \le c (x^\dagger, u^\dagger; y) \} \right).
\end{equation}
\end{itemize}
\end{remark}

\subsection{A minimization based approach using data inversion} \label{sec:Preliminaries}

In this section, we will introduce an approach to partially uncouple data dependence from reconstruction within the  minimization form (\ref{eq:minIP}), (\ref{eq:minIP_noise}) by using a right inverse operator $C^\ri$ of $C$ from (\ref{eq:observation}). To this end, we assume that the observation operator $C$ is \emph{linear} as is the case in many practical  problems. The linear operator $C^\ri$ is supposed to be a right inverse of $C$ in the sense that 
\begin{equation} \label{eq:condition_Cri}
(CC^\ri)|_{\Ima (C)} = \id_{\Ima (C)}
\end{equation}
where $\Ima(C) := \{C(u): u \in V\}$. If $V$, $Y$ are Hilbert spaces, one of the possible options to define $C^\ri$ is via the Moore-Penrose inverse operator $C^\dagger$, see \cite{EHN96}, with the domain $\mathcal{D} (C^\dagger) = \Ima (C) \oplus \Ima (C)^\bot \subset Y$. It is known that $C^\dagger$ is bounded iff $\Ima (C)$ is closed, which is also equivalent to $\mathcal{D} (C^\dagger) = Y$. Thus this approach is always applicable in case of a finite dimensional observation space $Y$.

Now, assuming that $C^\ri$ is well-defined and bounded on the entire space $Y$, by writing $u = C^\ri (y) + \hat u$ with $C(\hat u) = 0$, $y \in \Ima(C)$ we rephrase the problem (\ref{eq:model}), (\ref{eq:observation}) as
\begin{equation} \label{eq:model_observation_mixed}
\left\{
\begin{array}{l}
E(x, \hat u + C^\ri (y)) = 0, \\
(x, \hat u) \in X \times \Ker(C),
\end{array}
\right.
\end{equation}
with $\Ker(C) := \{ \hat u \in V: C(\hat u) = 0\}$. 
Thus we have decomposed the state into 
\begin{itemize}
\item[(a)] a part $C^\ri (y)$, which depends on the data and is therefore subject to noise, which, however, propagates into this part in a stable way; 
\item[(b)] a part $\hat u \in \Ker(C)$, which is data independent to which -- along with $x$ -- minimization will be applied in order to enforce the model equation to hold (approximately).
\end{itemize} 

Assuming that the priori information $\widetilde{\mathcal{R}} (x^\dagger, \hat u^\dagger) \le \rho$, with some radius $\rho>0$, and some functional $\widetilde{\mathcal{R}}: X \times V \to \overline{\mathbb{R}}$ is known, we consider 
\begin{equation} \label{eq:M_ad}
M_\ad (y) = \{ (x, \hat u) \in X \times \Ker(C): \widetilde{\mathcal{R}} (x, \hat u) \le \rho \}
\end{equation}
and 
\begin{equation} \label{eq:J}
\mathcal{J} (x, \hat u; y) = \mathcal{Q}_E (x, \hat u; y),
\end{equation}
where $\mathcal{Q}_E: X \times V \times Y \to \overline{\mathbb{R}}$ satisfies
\begin{equation} \label{eq:conditionQ}
\begin{split}
&\forall (x, \hat u, y) \in X \times V \times Y:\\
&\qquad \qquad \mathcal{Q}_E (x, \hat u; y) \ge 0 \text{\quad and \quad} \\
&\qquad \qquad \left( \forall \hat{u}\in \Ker(C)\, : E(x, \hat u + C^\ri (y)) = 0 \Leftrightarrow \mathcal{Q}_E (x, \hat u; y) = 0 \right).
\end{split}
\end{equation} 
With noisy data $y^\delta$ being available, we use the admissible set
\begin{equation} \label{eq:M_ad^delta} 
\begin{split}
M_\ad^\delta (y^\delta) &= \left\{ (x, \hat u) \in X \times V: \; \mathcal{S} (C(\hat u+C^\ri(y^\delta)), y^\delta) \le \tau \delta, \; \widetilde{\mathcal{R}} (x, \hat u) \le \rho \right\}
\end{split}
\end{equation}
with some $\tau>1$ and a discrepancy measure $\mathcal{S}: Y \times Y \to \overline{\mathbb{R}}$ such that
\begin{equation} \label{eq:conditionS_definiteness}
\forall y_1, y_2 \in Y: \qquad \mathcal{S} (y_1, y_2) \ge 0 \text{\quad and \quad} (\mathcal{S} (y_1, y_2) = 0 \Leftrightarrow y_1=y_2),
\end{equation}
\begin{equation} \label{eq:conditionS_delta}
\mathcal{S} (y, y^\delta) \le \delta,
\end{equation}
and
\begin{equation} \label{eq:conditionS_delta2}
\mathcal{S} (CC^\ri (y^\delta), y^\delta) < \tau \delta.
\end{equation}
This framework also includes discrepancy measures that are not necessarily defined by a metric or norm, for example, the Kullback Leibler divergence or the Bregman distance with respect to some proper convex functional.
In the special case of $\mathcal{S}$ being translation invariant, i.e., 
\[
\forall y_1,y_2,y_3\in Y: \ \mathcal{S}(y_1,y_2)=\mathcal{S}(y_1+y_3,y_2+y_3)
\]
and if $\Ima(C) \equiv Y$, i.e., $CC^\ri \equiv \id_Y$, the admissible set simplifies to
\begin{equation} \label{eq:M_ad^delta_reduced}
\begin{split}
M_\ad^\delta &= \left\{ (x, \hat u) \in X \times V: \; \mathcal{S} (C(\hat u), 0) \le \tau \delta, \; \widetilde{\mathcal{R}} (x, \hat u) \le \rho \right\}.
\end{split}
\end{equation}
To guarantee well-definedness, stability and convergence of the regularized minimization problems defined by \eqref{eq:minIP_noise} with \eqref{eq:J}, \eqref{eq:M_ad^delta}, we make some assumptions that are largely being used also in previous publications, see, e.g., \cite{Kal18, HKPS07, HW13, SGG+09}.

\begin{assumption} \label{ass:Maao}
Let a topology $\mathcal{T}$ and a norm $\| \cdot \|_B$ on $X \times V$ exist such that
\begin{enumerate}[label= (\roman*)]
\item $\widetilde{\mathcal{R}} (x^\dagger, \hat u^\dagger) \le \rho$;
\item $\mathcal{R}_j (x^\dagger, \hat u^\dagger) < \infty \text{\; and \;} \exists \underline{r} \in \mathbb{R}: \mathcal{R}_j \ge \underline{r}$ \; for all $j \in \{1, \dots, m\}$;
\item for all $z_1, z_2 \in Y$ and $c>0$, the sublevel set
\begin{equation*}
\begin{split}
&L_c =  \Big\{ (x, \hat u) \in X \times V: \\
& \quad \max \{ \mathcal{Q}_E (x, \hat u; z_1), \mathcal{R}_1 (x, \hat u), \dots, \mathcal{R}_m (x, \hat u), \widetilde{\mathcal{R}} (x, \hat u), \mathcal{S} (C(\hat u+C^\ri (z_2)), z_2) \} \le c \Big\}
\end{split}
\end{equation*}
is $\mathcal{T}$ compact and $\| \cdot \|_B$ bounded;
\item for all $z \in Y$, the maps $(x, \hat u) \mapsto \mathcal{Q}_E (x, \hat u; z)$, $(x, \hat u) \mapsto \mathcal{S} (C(\hat u+C^\ri (z)), z)$, $\mathcal{R}$ and $\widetilde{\mathcal{R}}$ are $\mathcal{T}$ lower semicontinuous;
\item the family of mappings $\big( \zeta \mapsto \mathcal{S}  (z+CC^\ri (\zeta), \zeta) \big)_{z \in Z}$ is uniformly continuous on $Z = \{ C(\hat u): \exists x \in X: \widetilde{\mathcal{R}} (x, \hat u) \le \rho \}$, i.e., 
$$\lim_{\zeta \to \zeta_0} \sup_{z \in Z} |\mathcal{S} (z+CC^\ri (\zeta), \zeta) - \mathcal{S} (z+CC^\ri (\zeta_0), \zeta_0)| = 0, \quad \forall \zeta_0 \in Y.$$
\end{enumerate}
\end{assumption}

\begin{remark} \label{rem:ass:Maao(v)}
If the admissible set (\ref{eq:M_ad^delta_reduced}) is used, regardless of whether or not $\Ima(C) \equiv Y$ holds, the condition (\ref{eq:conditionS_delta2}) and the Assumption \ref{ass:Maao}(v) can be dropped.
\end{remark}

Differently from \cite{Kal18}, we here have to deal with a model misfit functional $\mathcal{Q}$ that depends on the data and therefore make some continuity assumptions concerning this dependence. 
\begin{assumption}\label{ass:Maao_newcondition}
For all solution $(x^\dagger, \hat u^\dagger)$ to (\ref{eq:model_observation_mixed}) with the exact data $y$,
\begin{enumerate} [label = (\roman*)]
\item there exists a constant $\bar \delta >0$ such that
$$\lim_{z_2 \to z_1} \sup_{(x, \hat u) \in M} \left| \mathcal{Q}_E (x, \hat u; z_2) - \mathcal{Q}_E (x, \hat u; z_1) \right| = 0, \quad \forall z_1 \in Z,$$
where
\begin{itemize}
\item $Z$ is an open set that contains $\{ z \in Y: \mathcal{S} (y,z) \le \bar \delta \}$, for example, $Z = \cup_{\{ z \in Y: \mathcal{S} (y,z) \le \bar \delta \}} B_{\|\cdot\|_{Y}} (z,1)$,
\item $M = \cup_{z \in Z } (M_\ad^{\bar \delta} (z) \cap  M_z)$ with $M_z = \{ (x, \hat u): \mathcal{Q}_E (x, \hat u; z) \le c(x^\dagger, \hat u^\dagger; y) \}$ and $c(x^\dagger, \hat u^\dagger; y)$ as in Remark \ref{rem:adjust_ass:Kal18};
\end{itemize}
\item there exists a nondecreasing function $\gamma: [0,\infty] \to [0,\infty]$ such that
$$\mathcal{Q}_E (x^\dagger, \hat u^\dagger; z) - \mathcal{Q}_E (x^\dagger, \hat u^\dagger; y) \le \gamma (\mathcal{S} (y, z)), \quad \forall z \in Y.$$
\end{enumerate}
\end{assumption}

\begin{remark}\label{rem:Ass3}
In case of $\mathcal{Q}_E (x, \hat u; z)$ and $\mathcal{S}$ being defined by norms
\[
\mathcal{Q}_E(x, \hat u; z)=\sum_{i=1}^I \mathcal{Q}_i(x, \hat u; z)=\frac12\sum_{i=1}^I\|D_i(x, \hat u)C^{ri}z+b_i(x, \hat u)\|_W^2, 
\]
with linear operators $D_i(x, \hat u):V\to W$, as well as $\mathcal{S}(z_1,z_2)=\|z_1-z_2\|_Y$, as in the three examples of Section \ref{sec:ex}, we can easily verify Assumption \ref{ass:Maao_newcondition} with $\gamma(t)=c\cdot(1+t)\cdot t$ for some constant $c$, by using the estimate
\[
\begin{split}
&|\mathcal{Q}_E(x, \hat u; z_1)-\mathcal{Q}_E(x, \hat u; z_2)|\\
&\quad = \left|\sum_{i=1}^I (\sqrt{\mathcal{Q}_i(x, \hat u; z_1)}+\sqrt{\mathcal{Q}_i(x, \hat u; z_2)})
(\sqrt{\mathcal{Q}_i(x, \hat u; z_1)}-\sqrt{\mathcal{Q}_i(x, \hat u; z_2)})\right|
\\
&\quad\leq\tfrac{1}{\sqrt{2}}\sum_{i=1}^I (2\sqrt{\mathcal{Q}_E(x, \hat u; z_2)}+\tfrac{1}{\sqrt{2}}\|D_i(x, \hat u)C^{ri}(z_1-z_2)\|_W) \|D_i(x, \hat u)C^{ri}(z_1-z_2)\|_W\\
&\quad\leq\sum_{i=1}^I (\sqrt{2}\sqrt{\mathcal{Q}_E(x, \hat u; z_2)}+\tfrac12 \|D_i(x, \hat u)C^{ri}\|_{Y\to W} \mathcal{S} (z_1, z_2))\|D_i(x, \hat u)C^{ri}\|_{Y\to W} \mathcal{S} (z_1, z_2)
\,.
\end{split}
\]
\end{remark}

Under these assumptions we obtain the following result.

\begin{theorem} \label{the:main_theorem}
If Assumption \ref{ass:Maao}, Assumption \ref{ass:Maao_newcondition} and the conditions (\ref{eq:conditionQ}), (\ref{eq:conditionS_definiteness}), (\ref{eq:conditionS_delta}), (\ref{eq:conditionS_delta2}) are satisfied, then we achieve well-definedness, stability and convergence as follows. For any family of noisy data $(y^\delta)_{\delta \in (0,\bar \delta]}$,
\begin{enumerate}[label = (\alph*)]
\item $\forall \delta \in (0, \bar \delta]$, $\forall \alpha \in \mathbb{R}_+^m$, a minimizer of (\ref{eq:minIP_noise}) with (\ref{eq:J}), (\ref{eq:M_ad^delta}) exists;
\item for each $\delta \in (0, \bar \delta]$, for any sequence $(y_n)_{n \in \mathbb{N}} \subset Y$ with $y_n \to y^\delta$ in $Y$ as $n\to \infty$, the sequence of corresponding minimizers is $\| \cdot \|_B$ bounded.
\item If, additionally, the regularization parameter choice satisfies
$$ \alpha (\delta, y^\delta) \to 0 \text{\quad and \quad} \frac{\gamma(\delta)}{\alpha_j (\delta, y^\delta)} \le c, \forall j \in \{1,\dots,m\} \text{\qquad as \;} \delta \to 0,$$
for some $c \in \mathbb{R}$, then, as $\delta \to 0$, $y^\delta \to y$, the family of minimizers $(x^\delta_{\alpha (\delta, y^\delta)}, \hat u^\delta_{\alpha (\delta, y^\delta)})$ of (\ref{eq:minIP_noise}) with (\ref{eq:J}), (\ref{eq:M_ad^delta}) converges $\mathcal{T}$ subsequentially to a solution of the inverse problem with exact data, that is, it has a $\mathcal{T}$ convergent subsequence and the limit $(x^\dagger, \hat u^\dagger)$ forms a solution $(x^\dagger, \hat u^\dagger+C^{ri}y)$ of (\ref{eq:model}), (\ref{eq:observation}). 
\end{enumerate}
\end{theorem}

\begin{remark}\label{rem:uniqueness}
In Theorem \ref{the:main_theorem}, if the solution $(x^\dagger, \hat u^\dagger)$ is unique then $(x^\delta_{\alpha (\delta, y^\delta)}, \hat u^\delta_{\alpha (\delta, y^\delta)}) \xrightarrow{\mathcal{T}} (x^\dagger, \hat u^\dagger)$.
However, uniqueness is unlikely to hold with finite dimensional data.
\end{remark} 

\begin{proof} By (\ref{eq:conditionQ}), we see that $(x^\dagger, \hat u^\dagger)$ is a solution to (\ref{eq:model_observation_mixed}) if and only if it is a solution to (\ref{eq:minIP}). The rest of the proof consists of checking the items of Assumption \ref{ass:Kal18} where $u$ is replaced by $\hat u$.
\begin{itemize}
\item Assumption \ref{ass:Kal18}(a) follows from Assumption \ref{ass:Maao}(i) and 
(\ref{eq:conditionS_delta2}) which, due to the fact that $C\hat u^\dagger=0$, implies
\[
\mathcal{S} (C(\hat u^\dagger+C^\ri (y^\delta)), y^\delta)
=\mathcal{S} (CC^\ri (y^\delta)), y^\delta)<\tau\delta
\]
and therefore 
\[
\begin{split}
&\mathcal{S} (C(\hat u^\dagger+C^\ri (y_n)), y_n)\\
&\quad= \mathcal{S} (C(\hat u^\dagger+C^\ri (y^\delta)), y^\delta)
+\Bigl[\mathcal{S} (C(\hat u^\dagger+C^\ri (y_n)), y_n)-\mathcal{S} (C(\hat u^\dagger+C^\ri (y^\delta)), y^\delta)\Bigr]\\
&\quad\leq\tau\delta
\end{split}
\]
for all $n$ large enough, since the first term is strictly smaller than $\tau\delta$ by \eqref{eq:conditionS_delta2} and the term in brackets tends to zero by Assumption \ref{ass:Maao}(v).
\item Assumption \ref{ass:Kal18}(b) is exactly Assumption \ref{ass:Maao}(ii).
\item Assumption \ref{ass:Kal18}(c),(d) follow from the $\mathcal{T}$ compactness and $\| \cdot \|_B$ boundedness of $L_c$ as in Assumption \ref{ass:Maao}(iii).
\item Assumption \ref{ass:Kal18}(e) follows from the $\mathcal{T}$ lower semicontinuity of $\widetilde{\mathcal{R}}$ and of $(x, \hat u) \mapsto \mathcal{S} (C(\hat u+C^\ri(z)), z), \forall z \in Y$ according to Assumption \ref{ass:Maao}(iv).
\item Assumption \ref{ass:Kal18}(f) can be obtained by using the fact that $y\in \Ima(C)$ and $(x_n,\hat{u}_n)\in M_\ad^\delta (y_n)$ defined by \eqref{eq:M_ad^delta} with $y_n$ in place of $y^\delta$, as well as  Assumption \ref{ass:Maao}(iv),(v) to get
\begin{equation*}
\begin{split}
\mathcal{S} (C(\hat u_0)+y, y) &= \mathcal{S} (C(\hat u_0+C^{ri}(y)), y) \le \liminf_{n \to \infty} \mathcal{S} (C(\hat u_n+C^\ri(y)), y ) \\
&= \liminf_{n \to \infty} \big( \mathcal{S} (C(\hat u_n+C^\ri(z_n)), z_n) \\
& \qquad \qquad \quad + \mathcal{S} (C(\hat u_n+C^\ri(y)), y) - \mathcal{S} (C(\hat u_n+C^\ri(z_n)), z_n) \big) \\
& \le \limsup_{n \to \infty} \big( \tau \delta_n \ + \mathcal{S} (C(\hat u_n+C^\ri(y)), y ) - \mathcal{S} (C(\hat u_n+C^\ri(z_n)), z_n) \big) \\
& \le 0\,,
\end{split}
\end{equation*}
which implies $C(\hat u_0) =0$ by definiteness of $\mathcal{S}$ according to (\ref{eq:conditionS_definiteness}).
\item Assumption \ref{ass:Kal18}(g),(h) follows directly from Assumption \ref{ass:Maao}(iv).
\item Assumption \ref{ass:Kal18}(i),(j) adjusted by Remark \ref{rem:adjust_ass:Kal18} follow from Assumption \ref{ass:Maao_newcondition}(i), noting that $\mathcal{Q}_E (x^\dagger, \hat u^\dagger; y) = 0$, $ (x^\dagger, \hat u^\dagger) \in M$ and ``the suitable subset of $Y$'' in (\ref{eq:adjust_ass:Kal18_J(x,u,.)_continuous}) is $Z$.
\item The last one, Assumption \ref{ass:Kal18}(k), is verified by Assumption \ref{ass:Maao_newcondition}(ii) and (\ref{eq:conditionS_definiteness}), (\ref{eq:conditionS_delta}) as below
\begin{equation*}
\begin{split}
\frac{\mathcal{J} (x^\dagger, \hat u^\dagger; y^\delta) - \mathcal{J} (x^\dagger, \hat u^\dagger; y)}{\alpha_j (\delta, y^\delta)} &= \frac{\mathcal{Q} (x^\dagger, \hat u^\dagger; y^\delta) - \mathcal{Q} (x^\dagger, \hat u^\dagger; y)}{\alpha_j (\delta, y^\delta)}  \\
& \le \frac{\gamma(\mathcal{S} (y, y^\delta))}{\alpha_j (\delta, y^\delta)} \le \frac{\gamma(\delta)}{\alpha_j (\delta, y^\delta)} \le c,
\end{split}
\end{equation*}
and
\begin{equation*}
\begin{split}
\frac{ \mathcal{J} (x^\dagger, \hat u^\dagger; y) - \mathcal{J} (x^\delta, \hat u^\delta; y^\delta)}{\alpha_j (\delta, y^\delta)} & = \frac{0 - \mathcal{Q} (x^\delta, \hat u^\delta; y^\delta)}{\alpha_j (\delta, y^\delta)} \le 0.
\end{split}
\end{equation*}
\end{itemize}
\end{proof}

%

\section{Three application examples}\label{sec:ex}
Following up on a few briefly sketched examples in \cite{Kal18}, we here provide further evidence of the usefulness of minimization based formulation and regularization to real world problems.

To enable easy applicability of iterative minimization methods, we aim at working in Hilbert spaces $X$, $V$ for the design variables $x,u$. Also differentiability of $\mathcal{J}^\delta$ is an asset in this sense and holds for the examples below. On the constraints side, it is pointwise bounds that on one hand can be very efficiently implemented see, e.g., \cite{comp_minIP}, on the other hand are practically relevant in view of known a prior bounds on the searched for quantities. 
Morover, in the spirit of the Kohn-Vogelius functional, we strive for first order least squares formulations of the PDE models.
Another common feature of these examples and actually very characteristic in many real world applications is finite dimensionality of the data space, which enables application of the data inversion strategy from Section \ref{sec:Preliminaries}.
While in the examples of Sections \ref{sec:CEM-EIT}, \ref{sec:MPP}, the same PDE model comes with several different excitations, each of which leading to one of $I$ measurements (or measurement sets), there is only a single experiment carried out in the example of Section \ref{sec:SoundSources} and the measurement vector consists of observations at $L$ spatial points.

\subsection{Electrical impedance tomography with the complete electrode model} \label{sec:CEM-EIT}

Electrical impedance tomography (EIT) is a meanwhile well-established and well-researched imaging technology (see, e.g., the review \cite{Borcea2002} and the references therein), which seeks to recover the spatially varying electrical conductivity in the interior of an inhomogeneous object by means of low-frequency voltage-current measurements on its surface.
A minimization based formulation of this problem has already been devised in \cite{Knowles1998,KohnVogelius87,KM90} for the idealized measurement model that considers both voltages and currents as arbitrary (up to smoothness assumptions) functions on the boundary. A more realistic model of the electrodes has been developed in \cite{SCI92}, see also \cite{JXZ16}. Our aim here is to extend the variational formulation from \cite{Knowles1998,KohnVogelius87,KM90} to incorporate the complete electrode model (CEM) from \cite{SCI92,JXZ16}.

\subsubsection{The minimization form of the problem}

Let $\Omega$ be a domain in $\mathbb{R}^2$ with a boundary $\partial \Omega$ which is a closed and simple curve. Electrodes are placed on $\partial \Omega$ and are numbered in counterclockwise order from $e_1$ to $e_L$ with $\overline{e_{\ell_1}} \cap \overline{e_{\ell_2}} = \emptyset$ if $\ell_1 \ne \ell_2$. The starting point and the end point of $\ell$th electrode are denoted by $e_\ell^a$ and $e_\ell^b$ and the gap between $e_\ell$ and $e_{\ell+1}$ is denoted by $g_\ell$ (see Figure \ref{pic:CEM}). For convenience, we use the identities $e_{L+1} \equiv e_1$, $e_0 \equiv e_L$, $g_{L+1} \equiv g_1$, $g_0 \equiv g_L$. We also denote by $j_{i,\ell}$ the current applied on $e_\ell$ and by $v_{i,\ell}$ the voltage at the $\ell$th electrode  at the $i$th measurement, $i\in\{1,\ldots,I\}$. By the law of charge conservation $j_i=(j_{i,1}, \dots, j_{i,L})$ is an element of
$$\mathbb{R}^L_{\diamond}:= \left\{ (x_1,\dots,x_L) \in \mathbb{R}^L: \sum_{\ell=1}^{L} x_\ell =0 \right\}.$$
In addition, we can also normalize $v_i=(v_{i,1},\dots,v_{i,L})$ such that $v_i \in \mathbb{R}^L_{\diamond}$, $i\in\{1,\ldots,I\}$.

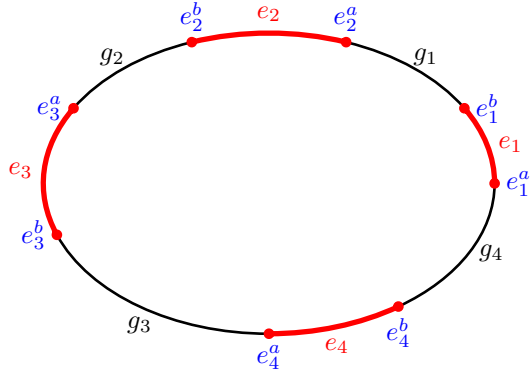
\begin{figure}[!h]
\begin{center}
\begin{tikzpicture} 
\draw (0,0) circle (3 and 2);
\coordinate (P) at ($(0, 0) + (30:3cm and 2cm)$);
\draw[black, line width=1pt] (P) arc (30:70:3cm and 2cm) node[pos=0.4, above]{$g_1$};
\coordinate (P) at ($(0, 0) + (110:3cm and 2cm)$);
\draw[black, line width=1pt] (P) arc (110:150:3cm and 2cm) node[pos=0.6, above]{$g_2$};
\coordinate (P) at ($(0, 0) + (200:3cm and 2cm)$);
\draw[black, line width=1pt] (P) arc (200:270:3cm and 2cm) node[pos=0.5, below]{$g_3$};
\coordinate (P) at ($(0, 0) + (305:3cm and 2cm)$);
\draw[black, line width=1pt] (P) arc (305:360:3cm and 2cm) node[pos=0.5, right]{$g_4$};

\draw[red, line width=2pt] 
(3,0) circle(1pt) node[blue, right] {$e_1^a$} arc (0:30:3cm and 2cm) circle(1pt) node[blue, right] {$e_1^b$} node[pos=0.5, right]{$e_1$};
\coordinate (P) at ($(0, 0) + (70:3cm and 2cm)$);
\draw[red, line width=2pt]
(P) circle (1pt) node[above, blue] {$e_2^a$} arc (70:110:3cm and 2cm) circle (1pt) node[above, blue] {$e_2^b$} node[pos=0.5, above]{$e_2$};
\coordinate (P) at ($(0, 0) + (150:3cm and 2cm)$);
\draw[red, line width=2pt]
(P) circle (1pt) node[left, blue] {$e_3^a$} arc (150:200:3cm and 2cm) circle (1pt) node[left, blue] {$e_3^b$} node[pos=0.5, left]{$e_3$};
\coordinate (P) at ($(0, 0) + (270:3cm and 2cm)$);
\draw[red, line width=2pt]
(P) circle (1pt) node[below, blue] {$e_4^a$} arc (270:305:3cm and 2cm) circle (1pt) node[below, blue] {$e_4^b$} node[pos=0.5, below]{$e_4$};
\end{tikzpicture}
\end{center}
\caption{Electrodes (in red) on the boundary with $L=4$. \label{pic:CEM}}
\end{figure}

The EIT-CEM problem is to find the conductivity $\sigma: \Omega \to \mathbb{R}$, $0 < \underline{\sigma} \le \sigma \le \overline{\sigma}, \text{\; a.e. in \;} \Omega$, satisfying
$$ \nabla \cdot J_i = 0, \quad \nabla^\bot \cdot E_i = 0, \quad J_i = \sigma E_i, \quad \text{in\;} \Omega,  \quad i=1,2,\dots, I$$
where $J_i:\Omega \to \mathbb{R}^2$ is the current density, $E_i:\Omega \to \mathbb{R}^2$ is the electric field and $\nabla^\bot$ is the 2-d rotation operator $\nabla^\bot = \left( - \frac{\partial}{\partial x_2},\frac{\partial}{\partial x_1} \right)$. In a similar way to \cite{Kal18,KM90,JXZ16}, by using potentials $\phi_i$ and $\psi_i$ for $J_i$ and $E_i$,
$$J_i = - \nabla^\bot \psi_i, \quad E_i = -\nabla \phi_i,$$
we write the problem in the form
\begin{subequations} \label{eq:EIT_CEM}
\begin{alignat}{2}
& \sqrt{\sigma}\nabla \phi_i - \frac{1}{\sqrt{\sigma}} \nabla^\bot \psi_i = 0 &&\quad \text{in \;} \Omega, \label{eq:EIT_CEM_eq} \\
& \phi_i + z_\ell \nabla^\bot \psi_i \cdot \nu = v_{i,\ell} &&\quad \text{on \;} e_\ell, \ \ell=1,2,\dots,L, \label{eq:EIT_CEM_constrain_el1} \\
& \int_{e_\ell} \nabla^\bot \psi_i \cdot \nu \dd s = j_{i,\ell} &&\quad \text{for \;} \ell=1,2,\dots,L, \label{eq:EIT_CEM_constrain_el2} \\
& \nabla^\bot \psi_i \cdot \nu = 0 &&\quad \text{on \;} \partial \Omega \backslash \cup_{\ell=1}^L e_\ell, \label{eq:EIT_CEM_constrain_outside_el}
\end{alignat}
\end{subequations}
for  $i= 1, 2, \dots, I$, where $\{z_\ell\}_{\ell=1}^{L}$ is the set of (known) positive contact impedances.
Here $I$ is the number of current patterns impressed via the $L$ electrodes.

Note that equations (\ref{eq:EIT_CEM}) contain only $\nabla^\bot \psi_i$, not $\psi_i$ itself. So adding or subtracting a constant from $\psi_i$ will have no effect on the problem and we can assume $\psi_i(e_1^a) = 0$.

By using (\ref{eq:EIT_CEM_constrain_el2}) and (\ref{eq:EIT_CEM_constrain_outside_el}) we see that $\psi_i(e_\ell^b) - \psi_i(e_\ell^a) = -j_{i,\ell}$ for $\ell \in\{1,\dots,L\}$, and $\psi_i(x) - \psi_i(e_{\ell}^b)=0$ for $x \in g_\ell,  \ell \in \{1,\dots,L\}$. It follows that $\psi_i(e_{\ell+1}^a) = \psi_i(e_{\ell}^b) = \psi_i(e_{\ell}^a) - j_{i,\ell}$ for $\ell \in \{1,\dots,L\}$ and then $\psi_i(e_{\ell+1}^a) = \psi_i(e_{\ell}^b) = \psi_i(e_1^a) - \sum_{k=1}^{l} j_{i,k} = -\sum_{k=1}^{l} j_{i,k}$  for $\ell \in \{1,\dots,L\}$. Therefore, the value of $\psi_i$ outside the electrodes is determined by $\psi_i(x) = -\sum_{k=1}^{l} j_{i,k}$ for $x \in g_\ell, \ell \in \{1,\dots,L\}$ or 
\begin{equation} \label{eq:condition_gl}
\psi_i|_{g_{\ell}} = -\sum_{k=1}^{l} j_{i,k}, \quad \mbox{ for all } \ell \in \{1,\dots,L\}.
\end{equation}
Next, by taking the integral from $e_\ell^a$ to $x$ on $e_\ell$ on both sides of (\ref{eq:EIT_CEM_constrain_el1}), noting that $\psi_i(e_\ell^a) = \psi_i|_{g_{\ell-1}}$, we get
\begin{equation} \label{eq:condition_el}
\int_{e_\ell^a}^x \phi_i \dd s - z_\ell (\psi_i(x) - \psi_i|_{g_{\ell-1}}) = v_{i,\ell} d_{\partial \Omega} (e_\ell^a, x), \quad \mbox{ for all }x \in e_\ell, \ \ell \in \{1,\dots,L\}
\end{equation}
where $d_{\partial \Omega} (x_1, x_2)$ is the length of $\partial \Omega$ from $x_1$ to $x_2$.

In view of (\ref{eq:condition_gl}), (\ref{eq:condition_el}), the function $\psi_i$ is constant on each $g_\ell$ and every function $C_\ell (\phi_i, \psi_i): e_\ell \to \mathbb{R}$,
\begin{equation} \label{eq:setting_Cl}
C_\ell(\phi_i, \psi_i)(x) := \frac{1}{d_{\partial \Omega} (e_\ell^a, x)} \left( \int_{e_\ell^a}^x \phi_i \dd s - z_\ell (\psi_i(x) - \psi_i|_{g_{\ell-1}}) \right), \quad \mbox{ for all } x \in e_\ell,
\end{equation}
is also constant on each $e_\ell$, so we can choose the spaces containing $\sigma$, $\overrightarrow{\phi} = (\phi_1, \dots, \phi_I)$, $ \overrightarrow{\psi} = (\psi_1, \dots, \psi_I)$ as 
\begin{equation} \label{eq:setting_xu}
\begin{split}
& \sigma \in X := L^2(\Omega), \\
& (\overrightarrow{\phi}, \overrightarrow{\psi}) \in V := \Big\{ (\overrightarrow{\phi}, \overrightarrow{\psi}) \in H^1 (\Omega)^{2I}: \forall i \in \{1, \dots, I\}, \; \forall \ell \in \{1, \dots, L\}, \\
& \qquad \qquad \qquad \qquad \psi_i (e_1^a) = 0, \qquad \psi_i|_{g_\ell}, C_\ell(\phi_i, \psi_i) \text{\; are constant,} \\
& \qquad \qquad \qquad \qquad \text{and \quad } \sum_{\ell=1}^{L} C_\ell (\phi_i, \psi_i) = 0\Big\}.
\end{split}
\end{equation}
We define the observation operator $C: V \to Y$ by
\begin{equation} \label{eq:setting_C}
C (\overrightarrow{\phi}, \overrightarrow{\psi}) = \left( \Big( \psi_i|_{g_{\ell-1}} - \psi_i|_{g_\ell} \Big)_{\ell=1}^L, \Big( C_\ell(\phi_i, \psi_i) \Big)_{\ell=1}^L \right)_{i=1}^I,
\end{equation}
where 
\begin{equation}
Y := \Big( \mathbb{R}_{\diamond}^L \times \mathbb{R}_{\diamond}^L \Big)^{I}
\end{equation} \label{eq:setting_Y}
and obviously $E (\sigma, \overrightarrow{\phi}, \overrightarrow{\psi}) = 
\Bigl(E_i(\sigma,\phi_i,\psi_i)\Bigr)_{i=1}^I \in L^2 (\Omega)^{2I}$ with 
\begin{equation} \label{eq:setting_P}
E_i(\sigma,\phi_i,\psi_i)=\sqrt{\sigma}\nabla \phi_i - \frac{1}{\sqrt{\sigma}} \nabla^\bot \psi_i .
\end{equation}

As a next step, we will determine the operator $C^\ri: Y \to V$ such that $C(C^\ri (y)) = y,$  for all 
\begin{equation}\label{yetaxi}
y = \left( (\eta_{i,\ell})_{\ell=1}^L, (\xi_{i,\ell})_{\ell=1}^L \right)_{i=1}^I \in Y, 
\end{equation}
that is, 
\begin{equation} \label{eq:formof_Cri}
C^\ri (y) = (C^\ri_{\overrightarrow{\phi}} (y), C^\ri_{\overrightarrow{\psi}} (y)) = (C^\ri_{\phi,1} (y), \dots, C^\ri_{\phi,I} (y), C^\ri_{\psi,1} (y), \dots, C^\ri_{\psi,I} (y))
\end{equation}
satisfying (\ref{eq:condition_gl}), (\ref{eq:condition_el}) where $(C^\ri_{\phi,i} (y), C^\ri_{\psi,i} (y))$ replaces $(\phi_i, \psi_i)$ and $(\eta_{i,\ell}, \xi_{i,\ell})$ replaces $(j_{i,\ell}, v_{i,\ell})$. Because of $C^\ri_{\psi,i} (y)|_{g_\ell} = -\sum_{k=1}^{l} \eta_{i,k}$, we choose $C^\ri_{\psi,i} (y)$ in the form
\begin{equation} \label{eq:formof_psi0}
C^\ri_{\psi,i} (y) = \sum_{\ell=1}^L \left( -\sum_{k=1}^{l} \eta_{i,k} \right) \psi_{0,\ell}
\end{equation} 
where $\psi_{0,\ell}(x) = 1$ if $x \in g_\ell$ and $\psi_{0,\ell} (x) = 0$ if $x \in g_k, k \ne l$, particularly, we choose $\psi_{0,\ell} \in H^1(\Omega)$ such that it satisfies the boundary conditions
\begin{equation} \label{eq:formof_psi0l}
\psi_{0,\ell} (x) = \left\{ \begin{array}{ll}
1, & \text{if \;} x \in g_\ell, \\
1 - d_{\partial \Omega} (e^b_\ell,x)/|e_\ell|, & \text{if \;} x \in e_\ell, \\
1 - d_{\partial \Omega} (e^a_{\ell+1},x)/|e_{\ell+1}|, & \text{if \;} x \in e_{\ell+1}, \\
0, & \text{if \;} x \in \partial \Omega \backslash (g_\ell \cup e_\ell \cup e_{\ell+1}),
\end{array} \right.
\end{equation}
where $|e_\ell|$  is the length of $e_\ell$. This is possible due to the fact that the function on the right hand side of (\ref{eq:formof_psi0l}) is in $H^{1/2} (\partial \Omega)$ together with the (inverse) Trace Theorem and the assumption that $\partial \Omega$ is Lipschitz. By some direct calculations on $C^\ri_{\psi,i} (y)$, it is easy to see that (recalling \eqref{yetaxi})
\begin{equation*}
\begin{split}
C^\ri_{\psi,i} (y) (x) &= -\sum_{k=1}^{l-1} \eta_{i,k} - \frac{\eta_{i,\ell}}{|e_\ell|} d_{\partial \Omega} (e^a_\ell, x) \\
&= C^\ri_{\psi,i} (y)|_{g_{i,l-1}} - \frac{\eta_{i,\ell}}{|e_\ell|} d_{\partial \Omega} (e^a_\ell, x), \quad \mbox{ for all } x \in e_\ell
\end{split}
\end{equation*}
and we get (by using (\ref{eq:condition_el}))
\begin{equation*}
\begin{split}
\int_{e_\ell^a}^x C^\ri_{\phi,i} (y) \dd s - z_\ell \left( C^\ri_{\psi,i} (y) (x) - C^\ri_{\psi,i} (y)|_{g_{\ell-1}} \right) &= \xi_{i,\ell} d_{\partial \Omega} (e_\ell^a, x), \quad \mbox{ for all } x \in e_\ell, \ \ell \in \{1,\dots,L\}
\end{split}
\end{equation*}
thus
\begin{equation*}
\begin{split}
\int_{e_\ell^a}^x C^\ri_{\phi,i} (y) \dd s &= \left( \xi_{i,\ell} - \frac{z_\ell}{|e_\ell|} \eta_{i,\ell} \right) d_{\partial \Omega} (e_\ell^a, x) \\
& = \left( \xi_{i,\ell} - \frac{z_\ell}{|e_\ell|} \eta_{i,\ell} \right) \int_{e_\ell^a}^x \dd s, \quad \mbox{ for all } x \in e_\ell, \ \ell \in \{1,\dots,L\}.
\end{split}
\end{equation*}
So $C^\ri_{\phi,i} (y)$ for $y$ according to \eqref{yetaxi} is defined as
\begin{equation} \label{eq:formof_phi0}
C^\ri_{\phi,i} (y) = \sum_{\ell=1}^L \left( \xi_{i,\ell} - \frac{z_\ell}{|e_\ell|} \eta_{i,\ell} \right) \phi_{0,\ell}
\end{equation}
where $\phi_{0,\ell}(x) = 1$ if $x \in e_\ell$ and $\phi_{0,\ell} (x) = 0$ if $x \in e_k, k \ne l$, particularly, we choose $\phi_{0,\ell} \in H^1(\Omega)$ satisfying the boundary conditions
\begin{equation} \label{eq:formof_phi0l}
\phi_{0,\ell} (x) = \left\{ \begin{array}{ll}
1, & \text{if \;} x \in e_\ell, \\
1 - d_{\partial \Omega} (e_\ell^b, x)/|g_\ell|, & \text{if \;} x \in g_\ell, \\
1 - d_{\partial \Omega} (e_\ell^a, x)/|g_{\ell-1}|, & \text{if \;} x \in g_{\ell-1}, \\
0, & \text{if \;} x \in \partial \Omega \backslash (e_\ell \cup g_\ell \cup g_{\ell-1}),
\end{array} \right.
\end{equation}
where $|g_\ell|$ is the length of $g_\ell$. This choice is also possible due to the same reason as for the existence of $\psi_{0,\ell}$ in (\ref{eq:formof_psi0l}). The so defined operator $C^\ri$ is linear and continuous as a mapping $C^\ri:Y\to V$.

\medskip

Now, the problem (\ref{eq:EIT_CEM}) is rewritten in the form (\ref{eq:model_observation_mixed}) as
\begin{equation} \label{eq:model_observation_mixed_EIT_CEM}
\left\{
\begin{array}{l}
\sqrt{\sigma} \nabla (\overrightarrow{\hat\phi} + C^\ri_{\overrightarrow{\phi}} (y)) - \frac{1}{\sqrt{\sigma}} \nabla^\bot (\overrightarrow{\hat\psi} + C^\ri_{\overrightarrow{\psi}} (y)) = 0, \\
(\sigma, \overrightarrow{\hat\phi}, \overrightarrow{\hat\psi}) \in L^\infty (\Omega) \times \Ker(C), \quad \mbox{ for all } i = 1,2,\dots,I,
\end{array}
\right.
\end{equation}
where $(C^\ri_{\overrightarrow{\phi}} (y), C^\ri_{\overrightarrow{\psi}} (y))$ satisfy (\ref{eq:formof_psi0}), (\ref{eq:formof_psi0l}), (\ref{eq:formof_phi0}), (\ref{eq:formof_phi0l}) and 
\begin{equation} \label{eq:KerC}
\begin{split}
\Ker(C) &= \big\{ (\overrightarrow{\hat\phi}, \overrightarrow{\hat\psi}) = (\hat \phi_1, \dots, \hat \phi_I, \hat \psi_1, \dots, \hat \psi_I,) \in V: \hat \psi_i|_{g_{\ell-1}} - \hat \psi_i|_{g_\ell} = 0, \\
&\qquad \qquad \text{\; and \;} C_\ell (\hat \phi_i, \hat \psi_i) = 0, \quad \forall \ell \in \{1, \dots, L\}, \forall i \in \{1, \dots, I\} \big\}.
\end{split}
\end{equation}

To regularize this problem using the theory from Section \ref{sec:Preliminaries}, we consider the cost function $\mathcal{J} (\sigma, \overrightarrow{\hat\phi}, \overrightarrow{\hat\psi}; y) = \mathcal{Q}_E (\sigma, \overrightarrow{\hat\phi}, \overrightarrow{\hat\psi}; y)$,
\[
\mathcal{Q}_E (\sigma, \overrightarrow{\hat\phi}, \overrightarrow{\hat\psi}; y) = 
\sum_{i=1}^I \mathcal{Q}_i(\sigma,\hat\phi_i, \hat\psi_i;y)
\]
where
\begin{equation} \label{eq:Q_P}
\begin{split}
&\mathcal{Q}_i(\sigma,\hat\phi_i, \hat\psi_i;y):=\begin{cases}\frac12\|q_i\|_{L^2(\Omega)}^2 \mbox{ if } q_i\in L^2(\Omega)\\
+\infty\mbox{ else}
\end{cases}\\
\mbox{with }
&q_i:=
E_i(\sigma,\hat\phi_i + C^\ri_{\phi,i} (y), \hat\psi_i + C^\ri_{\psi,i} (y))\\
&\hspace*{0.5cm}=
\sqrt{\sigma} \nabla (\hat \phi_i + C^\ri_{\phi,i} (y)) - \frac{1}{\sqrt{\sigma}} \nabla^\bot (\hat \psi_i + C^\ri_{\psi,i} (y))
\end{split}
\end{equation}
the regularization function
\begin{equation} \label{eq:R}
\mathcal{R} (\overrightarrow{\hat\phi}, \overrightarrow{\hat\psi}) = \frac{1}{2} \| (\overrightarrow{\hat\phi}, \overrightarrow{\hat\psi}) \|^2_{H^{s} (\Omega)^{2I}},
\end{equation}
for some $s>1$ the discrepancy measure
\begin{equation} \label{eq:S}
\mathcal{S} (y_1, y_2) = \| y_2 -y_1 \|_{\infty,\mathbb{R}^{2LI}}=\max\{|y_{2,j}-y_{1,j}|\, : \, j\in\{1,\ldots,2LI\},
\end{equation}
the function $\widetilde{\mathcal{R}}: L^2(\Omega) \to \overline{\mathbb{R}}$,
\begin{equation} \label{eq:R_tilde}
\widetilde{\mathcal{R}} (\sigma) = \left\| \sigma - \frac{\underline{\sigma} + \overline{\sigma}}{2} \right\|_{L^\infty (\Omega)}
\end{equation}
and the constant $\rho =\frac{\overline{\sigma} - \underline{\sigma}}{2}$. 
Note that therewith the constraint $\widetilde{\mathcal{R}}(\sigma) \le \rho$ is equivalent to $\underline{\sigma} \le \sigma \le \overline{\sigma}$ a.e. in $\Omega$.
The regularization term according to \eqref{eq:R} is only needed in order to prove existence of a minimizer. In particular, it allows to establish $\mathcal{T}$ lower semicontinuity of the regularized cost function with an appropriate topology $\mathcal{T}$ which would not be the case without this term.

If the accuracy of the current and voltage measurements is $\delta^{j}$ and $\delta^v$ respectively, i.e., $|j_{i,\ell}^\delta - j_{i,\ell}| \le \delta^j$ and $|v_{i,\ell}^\delta - v_{i,\ell}| \le \delta^v$ then the discrepancy measure between the perturbed data $y^\delta = \left( (j_{i,\ell}^\delta)_{\ell=1}^L, (v_{i,\ell}^\delta)_{\ell=1}^L \right)_{i=1}^I$ and the exact data $y = \left( (j_{i,\ell})_{\ell=1}^L, (v_{i,\ell})_{\ell=1}^L \right)_{i=1}^I$ is 
\begin{equation}\label{deltaEIT}
\mathcal{S} (y, y^\delta) = \| y^\delta - y \|_{\infty,\mathbb{R}^{2LI}} \le \max \{\delta^j, \delta^v\} =: \delta,
\end{equation}
where $\| \cdot \|_{\infty,\mathbb{R}^{2LI}}$ is the maximum norm and regularized minimization can be summarized as
 \begin{equation} \label{eq:minIP_CEM-EIT_noise}
\begin{split}
&\begin{split}\min \Bigg\{ \sum_{i=1}^I \frac{1}{2} \int_\Omega \left| \sqrt{\sigma} \nabla (\hat \phi_i + C^\ri_{\phi,i} (y^\delta)) - \frac{1}{\sqrt{\sigma}} \nabla^\bot (\hat \psi_i + C^\ri_{\psi,i} (y^\delta)) \right|^2 \dd x \\ + \frac{\alpha}{2} \| (\overrightarrow{\hat\phi}, \overrightarrow{\hat\psi}) \|^2_{H^{s} (\Omega) ^{2I}}:\end{split} \\
& \qquad (\sigma, \overrightarrow{\hat\phi}, \overrightarrow{\hat\psi}) \in \{ X \times V: \| C(\overrightarrow{\hat\phi}, \overrightarrow{\hat\psi}) \|_{\mathbb{R}^{2LI}} \le \tau \delta, \; \underline{\sigma} \le \sigma \le \overline{\sigma}, \text{\; a.e in \;} \Omega \} \Bigg\},
\end{split}
\end{equation}
where $X,V$ are defined in (\ref{eq:setting_xu}), $C$ in (\ref{eq:setting_C}), and $C^\ri$ in \eqref{eq:formof_psi0}--\eqref{eq:formof_phi0l}.

\begin{remark} \label{rem:CEM-EIT_choosing_M_ad^delta}
Because of $\Ima(C) \equiv Y$, the admissible sets in (\ref{eq:M_ad^delta}) and (\ref{eq:M_ad^delta_reduced}) are equal.
\end{remark}

\subsubsection{Convergence}
To obtain a convergence result, we define the topology $\mathcal{T}$ on $X \times V$ by
 \begin{equation} \label{eq:topologyT}
 (\sigma_n, \overrightarrow{\hat\phi}_n, \overrightarrow{\hat\psi}_n) \xrightarrow{\mathcal{T}} (\sigma, \overrightarrow{\hat\phi}, \overrightarrow{\hat\psi}) \Leftrightarrow 
\left \{
\begin{split}
\sigma_n \xrightharpoonup{*} \sigma \text{\; and \;} \frac{1}{\sigma_n} \xrightharpoonup{*} \frac{1}{\sigma} & \text{\; in \;} L^\infty (\Omega), \\
(\overrightarrow{\hat\phi}_n, \overrightarrow{\hat\psi}_n) \to (\overrightarrow{\hat\phi}, \overrightarrow{\hat\psi}) & \text{\; in \;} H^1 (\Omega)^{2I}, \\
(\overrightarrow{\hat\phi}_n, \overrightarrow{\hat\psi}_n) \rightharpoonup (\overrightarrow{\hat\phi}, \overrightarrow{\hat\psi}) & \text{\; in \;} H^{s} (\Omega)^{2I},
\end{split}
\right.
\end{equation}
for $s>1$ as in \eqref{eq:R} and the norm $\|\cdot\|_B$ is given by
\begin{equation} \label{eq:normB}
\| (\sigma, \overrightarrow{\hat\phi}, \overrightarrow{\hat\psi}) \|_B := \| \sigma \|_{L^\infty (\Omega)} + \| (\overrightarrow{\hat\phi}, \overrightarrow{\hat\psi}) \|_{H^{s} (\Omega)^{2I}} + \| C(\overrightarrow{\hat\phi}, \overrightarrow{\hat\psi}) \|_{\mathbb{R}^{2LI}}.
 \end{equation}
In addition, we also impose some priori conditions on $\sigma^\dagger$ and $(\overrightarrow{\hat\phi}^\dagger, \overrightarrow{\hat\psi}^\dagger)$ as follows
\begin{equation} \label{eq:priori_information_sigma}
\underline{\sigma} \le \sigma^\dagger \le \overline{\sigma}, \text{\quad a.e. in \;} \Omega
\end{equation}
and
\begin{equation} \label{eq:priori_information_phi_psi}
(\overrightarrow{\hat\phi}^\dagger, \overrightarrow{\hat\psi}^\dagger) \in H^{s} (\Omega)^{2I}.
\end{equation}

\begin{corollary} \label{cor:CEM-EIT}
Let (\ref{eq:priori_information_sigma}), (\ref{eq:priori_information_phi_psi}) hold. Then
\begin{enumerate} [label = (\roman*)]
\item (Existence of minimizers) For any $\alpha>0$, a minimizer of (\ref{eq:minIP_CEM-EIT_noise}) exists;
\item (Boundedness) For any sequence $(y_n)_{n \in \mathbb{N}} \subset Y$ with $y_n \to y^\delta$ in $Y$, the sequence of corresponding regularized minimizers is $\|\cdot\|_B$ bounded for $\|\cdot\|_B$ as in (\ref{eq:normB});
\item (Convergence) If additionally the choice of regularization parameter satisfies
$$\alpha(\delta, y^\delta) \to 0 \text{\quad and \quad} \frac{\delta^2}{\alpha(\delta, y^\delta)} \le c_0, \text{\qquad as \;} \delta \to 0, $$
then as $\delta \to 0$ in \eqref{deltaEIT}, $y^\delta \to y$, the family of minimizers $\left( \sigma_{\alpha(\delta, y^\delta)}^\delta, \overrightarrow{\hat\phi}_{\alpha(\delta, y^\delta)}^\delta, \overrightarrow{\hat\psi}_{\alpha(\delta, y^\delta)}^\delta \right)$ converges $\mathcal{T}$ subsequentially to a solution $(\sigma^\dagger, \overrightarrow{\hat\phi}^\dagger, \overrightarrow{\hat\psi}^\dagger)$ of the inverse problem (\ref{eq:model_observation_mixed_EIT_CEM}) with exact data $y$. 
\end{enumerate}
\end{corollary}

\begin{proof}
The proof proceeds by verification of Assumptions \ref{ass:Maao} and \ref{ass:Maao_newcondition}.
\begin{itemize}
\item Assumptions \ref{ass:Maao}(i),(ii) follow from (\ref{eq:priori_information_sigma}), (\ref{eq:priori_information_phi_psi}).
\item The set $L_c$ in Assumption \ref{ass:Maao}(iii) is clearly $\|\cdot\|_B$ bounded by definition of 
$\mathcal{R}$, $\widetilde{\mathcal{R}}$, $\mathcal{S}$.
Besides, for every sequence $(\sigma_n, \overrightarrow{\hat\phi}_n, \overrightarrow{\hat\psi}_n)$ in $L_c$,
\begin{itemize}
\item $(\overrightarrow{\hat\phi}_n, \overrightarrow{\hat\psi}_n)$, is bounded  and thus has a weakly convergent subsequence in $H^{s} (\Omega)^{2I}$, as well as a strongly convergent subsequence in $H^1 (\Omega)^{2I}$ by the compactness of the embedding operator from $H^{s} (\Omega)^{2I}$ to $H^1 (\Omega)^{2I}$;
\item the sequences $(\sigma_n)$, $\left(\frac{1}{\sigma_n}\right)$, which are bounded in $L^\infty (\Omega)$ due to boundedness of $\widetilde{\mathcal{R}} (\sigma_n)$, have weak-star convergent subsequences in $L^\infty (\Omega)$.
\end{itemize}
Thus we can extract a subsequence of $(\sigma_n, \overrightarrow{\hat\phi}_n, \overrightarrow{\hat\psi}_n)$ that is $\mathcal{T}$ convergent in $X \times V$. So Assumption \ref{ass:Maao}(iii) is verified.
\item The maps $\mathcal{R}$, $\widetilde{\mathcal{R}}$, $(\sigma, \overrightarrow{\hat\phi}, \overrightarrow{\hat\psi}) \mapsto \mathcal{S} (C(\overrightarrow{\hat\phi}, \overrightarrow{\hat\psi}), z)$ are obviously $\mathcal{T}$ lower semicontinuous by their definition. To verify Assumption \ref{ass:Maao}(iv), we only need to show $\mathcal{T}$ lower semicontinuity of $(\sigma, \overrightarrow{\hat\phi}, \overrightarrow{\hat\psi}) \mapsto \mathcal{Q}_E (\sigma, \overrightarrow{\hat\phi}, \overrightarrow{\hat\psi}; z)$ for all $z \in Y$. Indeed, we have
$\mathcal{Q}_E (\sigma, \overrightarrow{\hat\phi}, \overrightarrow{\hat\psi}; z) = \sum_{i=1}^I \mathcal{Q}_i (\sigma, \hat \phi_i, \hat \psi_i; z),$
where, cf. \eqref{eq:Q_P},
\begin{equation*}
\begin{split}
\mathcal{Q}_i (\sigma, \hat \phi_i, \hat \psi_i; z) 
&= \frac{1}{2} \int_\Omega \left( \sigma |\nabla \hat \phi_i|^2 + \frac{1}{\sigma} |\nabla^\bot \hat \psi_i|^2 - 2 \nabla \hat \phi_i \cdot \nabla^\bot \hat \psi_i \right) \dd x \\
& \quad + \int_\Omega \sigma \left( \nabla \hat \phi_i - \frac{1}{\sigma} \nabla^\bot \hat \psi_i \right) \cdot \left( \nabla C^\ri_{\phi,i} (z) - \frac{1}{\sigma} \nabla^\bot C^\ri_{\psi,i} (z) \right) \dd x \\
& \quad + \frac{1}{2} \int_\Omega \left( \sigma |\nabla C^\ri_{\phi,i} (z)|^2 + \frac{1}{\sigma} |\nabla^\bot C^\ri_{\psi,i} (z)|^2 - 2 \nabla C^\ri_{\phi,i} (z) \cdot \nabla^\bot C^\ri_{\psi,i} (z) \right) \dd x
\end{split}
\end{equation*}
and for all $(\sigma_n, \overrightarrow{\hat\phi}_n, \overrightarrow{\hat\psi}_n) \xrightarrow{\mathcal{T}} (\sigma, \overrightarrow{\hat\phi}, \overrightarrow{\hat\psi})$,
\begin{equation*}
\begin{split}
& |\mathcal{Q}_i (\sigma_n, \hat \phi_{n,i}, \hat \psi_{n,i}; z) - \mathcal{Q}_i (\sigma, \hat \phi_i, \hat \psi_i; z)| \\
& \quad \le \frac{1}{2} \left| \int_\Omega \left( \sigma_n |\nabla \hat \phi_{n,i}|^2 - \sigma |\nabla \hat \phi_i|^2 \right) \dd x \right| + \frac{1}{2} \left| \int_\Omega \left( \frac{1}{\sigma_n} |\nabla^\bot \hat \psi_{n,i}|^2 - \frac{1}{\sigma} |\nabla^\bot \hat \psi_i|^2 \right) \dd x \right| \\
& \quad \quad + \left| \int_\Omega \left( \nabla \hat \phi_{n,i} \cdot \nabla^\bot \hat \psi_{n,i} - \nabla \hat \phi_{i} \cdot \nabla^\bot \hat \psi_{i} \right) \dd x \right| \\
& \quad \quad + \left| \int_\Omega \left( \sigma_n \nabla \hat \phi_{n,i} - \sigma \nabla \hat \phi_i \right)  \cdot \left( \nabla C^\ri_{\phi,i} (z) - \frac{1}{\sigma} \nabla^\bot C^\ri_{\psi,i} (z) \right)\dd x \right| \\
& \quad \quad + \left| \int_\Omega \left( \frac{1}{\sigma_n} \nabla^\bot \hat \psi_{n,i} - \frac{1}{\sigma} \nabla^\bot \hat \psi_i \right)  \cdot \left( \nabla C^\ri_{\phi,i} (z) - \frac{1}{\sigma} \nabla^\bot C^\ri_{\psi,i} (z) \right)\dd x \right| \\
& \to 0,
\end{split}
\end{equation*}
because each term in the right hand side tends to $0$ as $n \to \infty$: \\
the first term
\begin{equation*}
\begin{split}
&\left| \int_\Omega \left( \sigma_n |\nabla \hat \phi_{n,i}|^2 - \sigma |\nabla \hat \phi_i|^2 \right) \dd x \right| \\
& \quad  = \left| \int_\Omega \sigma_n \left( | \nabla \hat \phi_{n,i}|^2 - |\nabla \hat \phi_i|^2 \right) \dd x + \int_\Omega (\sigma_n - \sigma) |\nabla \hat \phi_i|^2 \dd x \right| \\
& \quad \le \overline{\sigma} \int_\Omega \left| \left( \nabla \hat \phi_{n,i} + \nabla \hat \phi_i \right) \cdot \left( \nabla \hat \phi_{n,i} - \nabla \hat \phi_i \right) \right| \dd x + \left| \int_\Omega (\sigma_n - \sigma) |\nabla \hat \phi_i|^2 \dd x \right| \\
& \quad \le \overline{\sigma} \left\| \nabla \hat \phi_{n,i} + \nabla \hat \phi_i \right\|_{L^2 (\Omega)^2} \left\| \nabla \hat \phi_{n,i} - \nabla \hat \phi_i \right\|_{L^2 (\Omega)^2} + \left| \int_\Omega (\sigma_n - \sigma) |\nabla \hat \phi_i|^2 \dd x \right|
\end{split}
\end{equation*}
tends to $0$, since $\left\| \nabla \hat \phi_{n,i} + \nabla \hat \phi_i \right\|_{L^2 (\Omega)^2}$ is bounded, $\left\| \nabla \hat \phi_{n,i} - \nabla \hat \phi_i \right\|_{L^2 (\Omega)^2} \to 0$ as $n \to \infty$ and $(\sigma_n -\sigma) \xrightharpoonup{*} 0$ in $L^\infty (\Omega)$, $|\nabla \hat \phi_i|^2 \in L^1 (\Omega)$; 
note that at this point we need the strength of $\mathcal{T}$ inducing strong $H^1$ norm convergence of $\hat \phi_{n,i}$, as enabled by the regularization term \eqref{eq:R};\\
the second term can be estimated analogously;\\
the third term
\begin{equation*}
\begin{split}
& \left| \int_\Omega \left( \nabla \hat \phi_{n,i} \cdot \nabla^\bot \hat \psi_{n,i} - \nabla \hat \phi_{i} \cdot \nabla^\bot \hat \psi_{i} \right) \dd x \right| \\
& \quad \le \left| \int_\Omega \nabla \hat \phi_{n,i} \cdot \left( \nabla^\bot \hat \psi_{n,i} - \nabla^\bot \hat \psi_i \right) \dd x \right| + \left| \int_\Omega \left( \nabla \hat \phi_{n,i} - \nabla \hat \phi_i \right) \cdot \nabla^\bot \hat \psi_i \dd x \right| \\
& \quad \le \left\| \nabla \hat \phi_{n,i} \right\|_{L^2 (\Omega)^2} \left\| \nabla^\bot \hat \psi_{n,i} - \nabla^\bot \hat \psi_i \right\|_{L^2 (\Omega)^2} + \left\| \nabla \hat \phi_{n,i} - \nabla \hat \phi_i \right\|_{L^2 (\Omega)^2} \left\| \nabla^\bot \hat \psi_i \right\|_{L^2 (\Omega)^2}
\end{split}
\end{equation*}
tends to $0$, since $\nabla^\bot \hat \psi_{n,i} \to \nabla^\bot \hat \psi_i$ and $\nabla \hat \phi_{n,i} \to \nabla \hat \phi_i$ in $L^2 (\Omega)^2$; \\
the fourth term
\begin{equation*}
\begin{split}
& \left| \int_\Omega \left( \sigma_n \nabla \hat \phi_{n,i} - \sigma \nabla \hat \phi_i \right)  \cdot \left( \nabla C^\ri_{\phi,i} (z) - \frac{1}{\sigma} \nabla^\bot C^\ri_{\psi,i} (z) \right)\dd x \right| \\
& \quad \le \left| \int_\Omega \left( \left( \sigma_n - \sigma \right) \nabla \hat \phi_{n,i} + \sigma \left( \nabla \hat \phi_{n,i} - \nabla \hat \phi_i \right) \right)  \cdot \left( \nabla C^\ri_{\phi,i} (z) - \frac{1}{\sigma} \nabla^\bot C^\ri_{\psi,i} (z) \right)\dd x \right| \\
& \quad \le \left| \int_\Omega \left( \sigma_n - \sigma \right) \nabla \hat \phi_{n,i} \cdot \left( \nabla C^\ri_{\phi,i} (z) - \frac{1}{\sigma} \nabla^\bot C^\ri_{\psi,i} (z) \right) \dd x \right| \\
& \quad \qquad + \overline{\sigma} \left\| \nabla \hat \phi_{n,i} - \nabla \hat \phi_i \right\|_{L^2 (\Omega)^2} \left\| \nabla C^\ri_{\phi,i} (z) - \frac{1}{\sqrt{\sigma}} \nabla^\bot C^\ri_{\psi,i} (z) \right\|_{L^2 (\Omega)^2}
\end{split}
\end{equation*}
tends to $0$, since $(\sigma_n -\sigma) \xrightharpoonup{*} 0$ in $L^\infty (\Omega)$, $\nabla \hat \phi_{n,i} \cdot \left( \nabla C^\ri_{\phi,i} (z) - \frac{1}{\sigma} \nabla^\bot C^\ri_{\psi,i} (z) \right) \in L^1 (\Omega)$ and $\left\| \nabla \hat \phi_{n,i} - \nabla \hat \phi_i \right\|_{L^2 (\Omega)^2} \to 0$;\\
the fifth term again works analogously.
\item Assumption \ref{ass:Maao}(v) is verified by Remark \ref{rem:ass:Maao(v)}.
\item For Assumption \ref{ass:Maao_newcondition}, we employ Remark \ref{rem:Ass3} with $W=L^2(\Omega)$, $D_i(\sigma, \phi_i,\psi_i)C^\ri(\eta,\xi)=\sqrt{\sigma}\nabla C^\ri_{\phi,i} (\eta,\xi)- \frac{1}{\sqrt{\sigma}} \nabla^\bot C^\ri_{\psi,i} (\eta,\xi)$, cf. \eqref{eq:formof_psi0}, \eqref{eq:formof_phi0}, \eqref{eq:Q_P}, and the estimate 
\[
\begin{split}
&\|D_i(\sigma, \phi_i,\psi_i)C^\ri(\eta,\xi)\|_{L^2(\Omega)}\\
&\leq
\sqrt{\overline{\sigma}}\sum_{\ell=1}^L\left(|\xi_{i,\ell}|+\frac{|z_\ell|}{|e_\ell|}|\eta_{i,\ell}|\right)\|\nabla\phi_{0,\ell}\|_{L^2(\Omega)}
+\frac{1}{\sqrt{\underline{\sigma}}} \sum_{\ell=1}^L\sum_{k=1}^\ell |\eta_{i,k}| \|\nabla^\bot\psi_{0,\ell}\|_{L^2(\Omega)}\,.
\end{split}
\]
\end{itemize}
\end{proof}

\subsection{Identification of a nonlinear magnetic permeability from magnetic flux measurements} \label{sec:MPP}
The magnetic permeability, i.e., the factor relating the magnetic field strength $H$ to the magnetic flux density $B$, often exhibits a nonlinear behaviour. This is the case in particular in the presence of large field strengths as typical for actuator applications. To determine this nonlinear relation from magnetic flux measurements, either a very specific experimental geometry allowing for model simplifications needs to be employed or the full PDE model incorporating Maxwell's equations has to be taken into account, cf. \cite{KKR03} and the references therein.
We here follow the latter approach and derive a minimization based formulation for the corresponding inverse problem.

\subsubsection{The minimization form of the problem}
The magnetic field $H$ and the magnetic flux density $B$ are related by $B =\mu H$, where $\mu$ is the magnetic permeability. In the presence of large field strengths, it exhibits a dependence on the magnetic field strength $|H|$, so $\mu = \mu(|H|)$ and $B =\mu(|H|) H$.
The typical experimental setup for determining the $B-H$ curve (see Figure \ref{fig:magn} left) or equivalently the permeability curve, is depicted in Figure \ref{fig:magn} right.
\begin{figure}[ht]
\begin{center}
\includegraphics[width=0.39\textwidth]{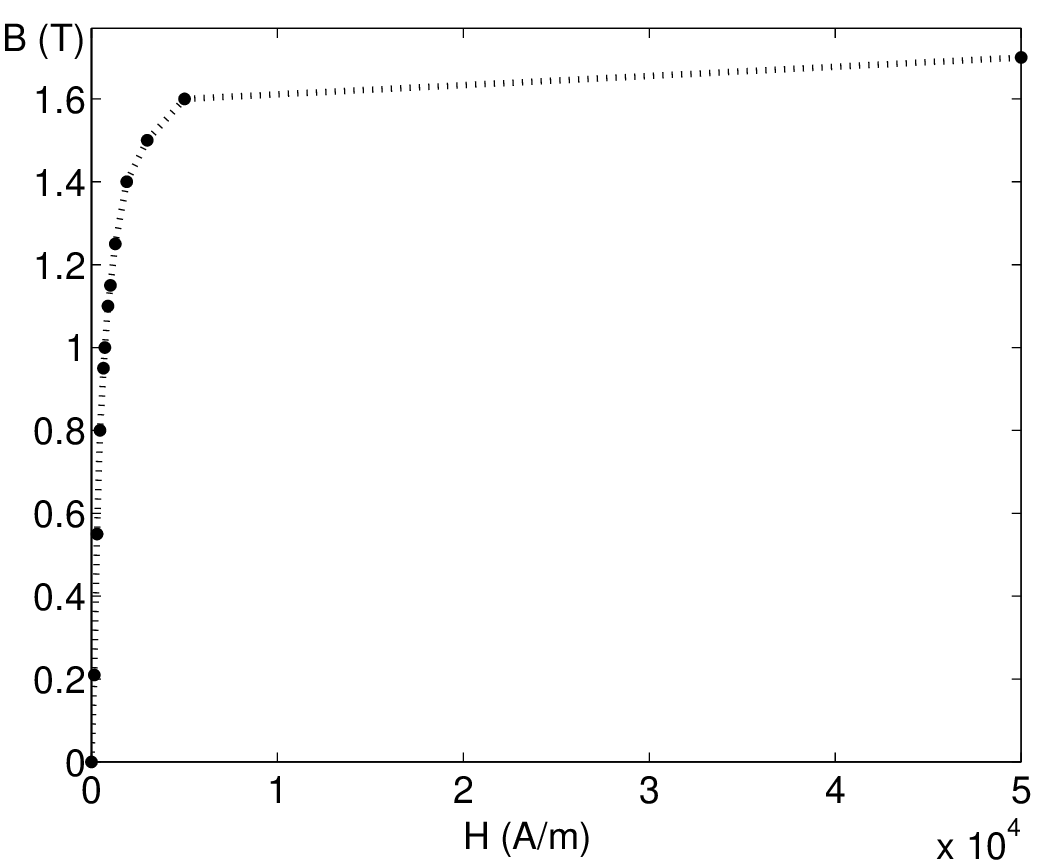}
\hspace*{0.05\textwidth}
\includegraphics[width=0.54\textwidth]{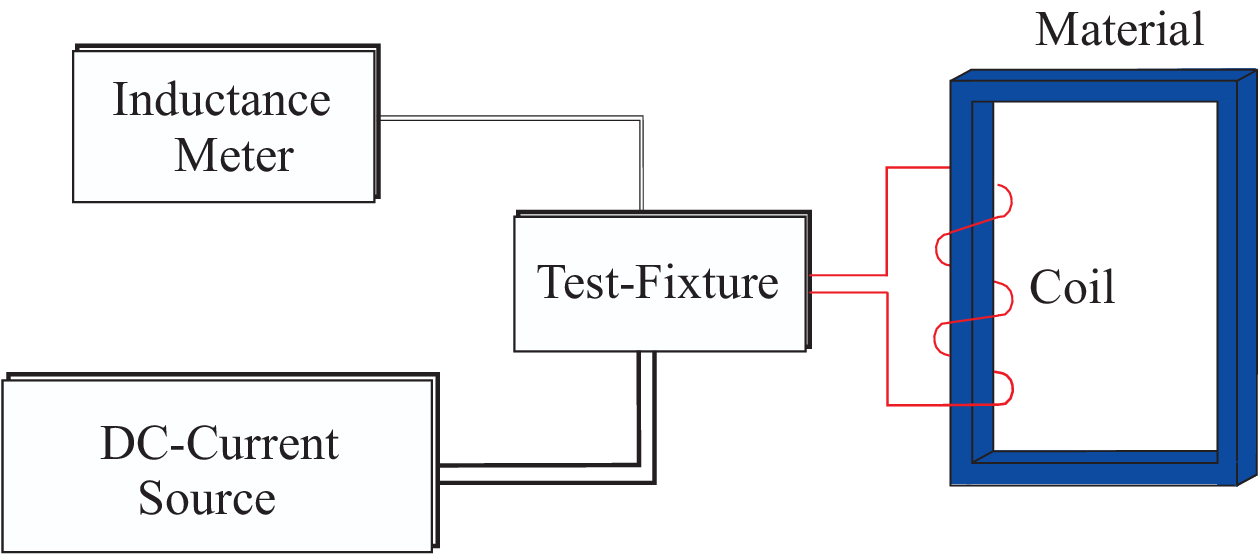}
\caption{Typical $B$--$H$-curve shape (left) and measurement setup (right)} \label{fig:magn}
\end{center}
\end{figure}

The current $J_i^{\rm imp}$ impressed by the excitation coil generates a magnetic field $H_i$ and a magnetic flux density $B_i$. They must satisfy Maxwell's equations
$$\nabla \cdot B_i = 0, \quad \nabla \times H_i = J_i^{{\rm imp}}  \quad  \text{ in} \; \Omega\subset \mathbb{R}^3, \quad i=1,\cdots,I$$
and the relation $$B_i = \mu(|H_i|) H_i.$$
For a given sequence of impressed currents $(J_i^\imp)_{i=1,\dots,I}$, we can measure the magnetic fluxes $\overrightarrow{\Phi} = (\Phi_i)_{i=1,\dots,I}$, with
$$\Phi_i = \frac{1}{h} \int_{\Omega_c} B_i \cdot \nu \dd x \quad i=1,\dots,I,$$
where $\Omega_c := \Gamma_c \times [0,h]$ is the region covered by the excitation coil, $\Gamma_c$  is its cross section, $h$ is the coil height and $\nu$ is the unit vector parallel to the coil axis, see \cite{KKR03} for a more detailed description of the measurement setup. By using the vector and scalar potentials $\overrightarrow{A} = (A_i)_{i=1,\dots,I}$, $\overrightarrow{\psi} = (\psi_i)_{i=1,\dots,I}$, with
$$B_i = \nabla \times A_i,  \quad H_i = \nabla \psi_i + A_i^J, \quad i=1, \cdots,I$$
(where $\nabla \times A^J_i = J_i^{{\rm imp}}, i=1,\dots,I$), the problem can be rewritten as 
\begin{equation} \label{eq:permeability_problem}
\left\{
\begin{split}
&\sqrt{\mu_i} (\nabla \psi_i + A_i^J) - \frac{1}{\sqrt{\mu_i}} \nabla \times A_i =0 \text{\quad in~} \Omega, \\
&\psi_i|_{\partial \Omega}=0, \quad \nu \times A_i|_{\partial \Omega}=0, \quad C(\overrightarrow{A},\overrightarrow{\psi}) = \overrightarrow{\Phi}, \quad i=1,\dots,I,
\end{split}
\right.
\end{equation}
where we abbreviate 
\begin{equation} \label{eq:mu}
\mu_i := \mu (|H_i|) = \mu (|\nabla \psi_i + A_i^J|),
\end{equation}
the spaces containing $\mu$, $\overrightarrow  A$, $\overrightarrow  \psi$ are defined by
\begin{equation} \label{eq:permeability_mu_A_phi}
\begin{split}
\mu \in X &:= L^2 ([0,\infty)), \\
(\overrightarrow  A, \overrightarrow  \psi) \in V &:= H_0 (\curl, \Omega)^I \times H_0^1 (\Omega)^I,
\end{split}
\end{equation}
and $C: V \to Y$, ($Y := \mathbb{R}^{I}$ with the maximum norm), is the observation operator defined by
\begin{equation} \label{eq:permeability_C}
C(\overrightarrow{A}, \overrightarrow{\psi}):=  \left( \frac{1}{h} \int_{\Omega_c} B_i \cdot \nu \dd x \right)_{i=1,\dots,I} = \left( \frac{1}{h} \int_{\Omega_c} (\nabla \times A_i) \cdot \nu \dd x \right)_{i=1,\dots,I}.
\end{equation}

To apply the theory from Section \ref{sec:Preliminaries}, we define a right inverse $C^\ri$ of $C$ in the form 
\begin{equation} \label{eq:permeability_Cri}
C^\ri (\overrightarrow  \Phi) = (C^\ri_A (\overrightarrow  \Phi), C^\ri_\psi (\overrightarrow  \Phi)) = \left( (C^\ri_{A_i} (\overrightarrow  \Phi))_{i=1}^I, (C^\ri_{\psi_i} (\overrightarrow  \Phi))_{i=1}^I \right) \in V,
\end{equation}
for all $\overrightarrow  \Phi \in \mathbb{R}^I$. Because of the absence of the component $\overrightarrow  \psi$ in the right hand side of (\ref{eq:permeability_C}), we can choose
\begin{equation} \label{eq:permeability_Cri_psi}
C^\ri_{\psi_i} (\overrightarrow  \Phi)=0, \quad \forall i \in \{1, \dots, I\}.
\end{equation}
Without loss of generality, we can assume that $\nu = (0,0,1)$ and choose the component $C^\ri_A (\overrightarrow  \Phi) = (C^\ri_{A_i} (\overrightarrow  \Phi))_{i=1}^I$ in (\ref{eq:permeability_Cri}) as follows
\begin{equation} \label{eq:permeability_Cri_A}
C^\ri_{A_i} (\overrightarrow  \Phi) (x_1, x_2, x_3) = \left( 0, \frac{\Phi_i}{|\Gamma_c|} x_1, 0 \right)^T, \quad \forall i \in \{1, \dots, I\},
\end{equation}
where $|\Gamma_c|$ is the area of $\Gamma_c$. It is easy to check that $C(C^\ri(\overrightarrow  \Phi)) = \overrightarrow  \Phi$ for all $\overrightarrow  \Phi \in \mathbb{R}^I$.

\begin{remark} \label{re:permeability_choiceOfCri}
It is evident that there are many ways to choose $C^\ri$ instead of (\ref{eq:permeability_Cri}), (\ref{eq:permeability_Cri_psi}), (\ref{eq:permeability_Cri_A}). For example, we can choose $\widetilde{C^\ri} (\overrightarrow{ \Phi}) = C^\ri (\overrightarrow{ \Phi}) + u$ with some $u \in \Ker (C)$. In this paper, we only focus on demonstrating the applicability of our new approach from Section \ref{sec:Preliminaries} so the question of how to optimally choose $C^\ri$ is not considered in more detail here. This remark also applies to the choice of $C^\ri$ (\ref{eq:formof_Cri}), (\ref{eq:formof_psi0}), (\ref{eq:formof_phi0}) in the previous section.
\end{remark}

Now we rewrite the problem (\ref{eq:permeability_problem}) in the form of (\ref{eq:model_observation_mixed}) as
\begin{equation} \label{eq:permeability_problem_mixed}
\left\{
\begin{split}
&E_i(\mu, \hat \psi_i,A_i) =
\sqrt{\mu_i} (\nabla \hat \psi_i + A_i^J) - \frac{1}{\sqrt{\mu_i}} \nabla \times (\hat A_i + C^\ri_{A_i} (\overrightarrow  \Phi)) = 0 \text{\quad in~} \Omega, \\
&(\mu, \overrightarrow{\hat A}, \overrightarrow{\hat \psi}) \in L^\infty (\Omega) \times \Ker(C) = L^\infty (\Omega) \times \{ (\overrightarrow{\hat A}, \overrightarrow{\hat \psi}) \in V: C(\overrightarrow{\hat A}, \overrightarrow{\hat \psi}) = 0 \},
\end{split}
\right.
\end{equation}
where $\overrightarrow{\hat A} = (\hat A_i)_{i=1}^I$, $\overrightarrow{\hat \psi} = (\hat \psi_i)_{i=1}^I$, and $\mu_i$ is defined by \eqref{eq:mu}. 
Additionally, we consider the cost function $\mathcal{J} (\mu, \overrightarrow{\hat A}, \overrightarrow{\hat \psi}; \overrightarrow  \Phi) = \mathcal{Q}_E(\mu, \overrightarrow{\hat A}, \overrightarrow{\hat \psi}; \overrightarrow  \Phi)$, 
\begin{equation} \label{eq:permeability_setting_Q}
\mathcal{Q}_E(\mu, \overrightarrow{\hat A}, \overrightarrow{\hat \psi}; \overrightarrow  \Phi) =  \sum_{i=1}^I 
\mathcal{Q}_i(\mu, {\hat A}_i, {\hat \psi}_i; \overrightarrow  \Phi)
\end{equation}
where 
\begin{equation} \label{eq:permeability_setting_Q_i}
\begin{split}
\mathcal{Q}_i(\mu, {\hat A}_i, {\hat \psi}_i; \overrightarrow  \Phi)=\begin{cases}\frac12\|q_i\|_{L^2(\Omega)}^2 \mbox{ if } q_i\in L^2(\Omega)^3\\
+\infty\mbox{ else}
\end{cases}\\
\mbox{with }q_i= \sqrt{\mu_i} (\nabla \hat \psi_i + A_i^J) - \frac{1}{\sqrt{\mu_i}} \nabla \times (\hat A_i + C^\ri_{A_i} (\overrightarrow  \Phi)),
\end{split}
\end{equation}
the discrepancy measure
\begin{equation} \label{eq:permeability_setting_S}
\mathcal{S} (y_1, y_2) = \| y_2 -y_1 \|_{\infty,\mathbb{R}^{I}},
\end{equation}
the regularization functionals
\begin{equation} \label{eq:permeability_setting_R}
\mathcal{R}: X \times V \to \overline{\mathbb{R}}, \qquad \mathcal{R} (\mu, \overrightarrow{\hat A}, \overrightarrow{\hat \psi}) = \frac{1}{2} \left( \| \overrightarrow{\hat A} \|_{H^s(\Omega)^{3I}}^2 + \|\overrightarrow{\hat \psi} \|_{H^s(\Omega)^{I}}^2 \right),
\end{equation}
where $s>1$ is given, and
\begin{equation} \label{eq:permeability_setting_R_tilde}
\widetilde{\mathcal{R}}: X \times V \to \overline{\mathbb{R}}, \qquad \widetilde{\mathcal{R}}(\mu,\overrightarrow{\hat A}, \overrightarrow{\hat \psi}) = \max\limits_{i=1,\dots,I}\left\| \mu_i - \frac{1}{2}(\overline{\mu}+\underline{\mu}) \right\|_{L^\infty(\Omega)},
\end{equation}
where $\overline{\mu}, \underline{\mu}$ are upper and lower bounds for the permeability ($0<\underline{\mu}<\overline{\mu}$). 

If the measurement data $\overrightarrow  \Phi^\delta = (\Phi_i^\delta)_{i=1}^I$ is obtained with the accuracy of $\delta$, i.e. 
$$\mathcal{S} (\overrightarrow  \Phi, \overrightarrow  \Phi^\delta) = \| \overrightarrow  \Phi^\delta - \overrightarrow  \Phi \|_{\mathbb{R}^I} = \max \{ |\Phi_i^\delta - \Phi_i |: i\in \{1,\dots,I\} \} \le \delta,$$
then the problem (\ref{eq:minIP_noise}) becomes
\begin{equation} \label{eq:permeability_minIP_noise}
\begin{split}
&\begin{split}\min \Bigg\{ \frac{1}{2} \sum_{i=1}^I \int_\Omega \left| \sqrt{\mu_i} (\nabla \hat \psi_i + A_i^J) - \frac{1}{\sqrt{\mu_i}} \nabla \times (\hat A_i + C^\ri_{A_i} (\overrightarrow  \Phi^\delta)) \right|^2 \dd x \\ + \frac{\alpha}{2} \left( \| \overrightarrow{\hat A} \|_{H^s(\Omega)^{3I}}^2 + \|\overrightarrow{\hat \psi} \|_{H^s(\Omega)^{I}}^2 \right):\end{split} \\
& \qquad (\mu, \overrightarrow{\hat A}, \overrightarrow{\hat \psi}) \in \{ X \times V: \| C(\overrightarrow{\hat A}, \overrightarrow{\hat \psi}) \|_{\mathbb{R}^{I}} \le \tau \delta, \; \underline{\mu} \le \mu_i \le \overline{\mu}, \text{\; a.e in \;} \Omega, \; \forall i \} \Bigg\}
\end{split}
\end{equation}
with $X,V$ as in \eqref{eq:permeability_mu_A_phi} and $\mu_i$ as in \eqref{eq:mu}.
In the next subsection, we will prove well-definedness, stability, and convergence of the minimizers of problem (\ref{eq:permeability_minIP_noise}).

\subsubsection{Convergence}
The topology $\mathcal{T}$ in Assumption \ref{ass:Kal18} can be chosen as
\begin{equation} \label{eq:permeability_topology_T}
(\mu^n, \overrightarrow{A}^n, \overrightarrow{\psi}^n) \xrightarrow{\mathcal{T}} (\mu, \overrightarrow{A}, \overrightarrow{\psi}) \Leftrightarrow 
\left\{ \begin{split}
& \begin{split} \mu^n(|H^n_i|) \xrightharpoonup{*} \mu(|H_i|) \text{\; and \;}  \frac{1}{\mu^n(|H^n_i|)} \xrightharpoonup{*} \frac{1}{\mu(|H_i|)}\\ \text{\; in\;} L^{\infty}(\Omega), \quad i=1,\dots,I,\end{split} \\
& \overrightarrow{A}^n \to \overrightarrow{A} \text{\; in \;} H(\curl; \Omega)^I, \quad \overrightarrow{\psi}^n \to \overrightarrow{\psi} \text{\; in \;} H^1(\Omega)^I, \\
& \overrightarrow{A}^n \rightharpoonup \overrightarrow{A} \text{\; in \;} H^s(\Omega)^{3I}, \quad \overrightarrow{\psi}^n \rightharpoonup \overrightarrow{\psi} \text{\; in \;} H^s(\Omega)^I, \\
\end{split} \right.
\end{equation}
for $s>1$ as in \eqref{eq:permeability_setting_R}
where $H^n_i:=\nabla \psi^n_i + A^{J}_i, \; H_i:=\nabla \psi_i + A^J_i, \;i=1,\dots,I$, and the norm $\|\cdot\|_B$ is defined by
\begin{equation} \label{eq:permeability_norm_B}
\| (\mu, \overrightarrow{A}, \overrightarrow{\psi}) \|_B := \max \limits_{i=1,\dots,I} \| \mu(|H_i|) \|_{L^{\infty}(\Omega)} + \| \overrightarrow{A} \|_{H^s(\Omega)^{3I}} + \| \overrightarrow{\psi} \|_{H^s(\Omega)^I} + \| C(\overrightarrow{A},\overrightarrow{\psi})\|_{\mathbb{R}^I}.
\end{equation}
Note that convergence in the combined topology \eqref{eq:permeability_topology_T} is different from weak* $L^\infty([0,\infty))$ convergence of $\mu^n$, see the appendix for more details. As a matter of fact, we can only expect $\mu$ to be determined on the union of the ranges of the magnetic fields corresponding to the exact solution, i.e., on $\bigcup_{i=1}^I\{|\nabla\psi_i^\dagger(x)+A^{J}_i(x)|\, : \, x\in\Omega\}=:\mathcal{D}(\overrightarrow{\psi^\dagger})\subseteq[0,\infty)$. This domain is a priori unknown, though, so we consider $\mu$ as a function on $[0,\infty)$ and, according to the first line in \eqref{eq:permeability_topology_T} trust its reconstruction only on $\mathcal{D}(\overrightarrow{\psi})\subseteq[0,\infty)$, where $\overrightarrow{\psi}$ is the set of reconstructed scalar potentials. 

Similarly to conditions (\ref{eq:priori_information_sigma}), (\ref{eq:priori_information_phi_psi}), we also need
\begin{equation} \label{eq:permeability_priori_information_mu}
\underline{\mu} \le \mu^\dagger(|H^\dagger_i|) \le \overline{\mu}, \quad \text{\; a.e. in \;} \Omega, \quad i\in\{1,\ldots,I\}
\end{equation}
and
\begin{equation} \label{eq:permeability_priori_information_A_psi}
\left( \overrightarrow{\hat A}^\dagger,\overrightarrow{\hat \psi}^\dagger \right) \in H^s (\Omega)^{3I} \times H^s (\Omega)^I
\end{equation}
to verify Assumptions \ref{ass:Maao}, \ref{ass:Maao_newcondition} and get Corollary \ref{cor:permeability} as follows.

\begin{corollary} \label{cor:permeability}
Let (\ref{eq:permeability_priori_information_mu}), (\ref{eq:permeability_priori_information_A_psi}) hold.
\begin{itemize}
\item[(i)] (Existence of minimizers) Then for any $\alpha>0$, a minimizer of (\ref{eq:permeability_minIP_noise}) exists;
\item[(ii)] (Boundedness) and for any sequence $(\overrightarrow  \Phi^n)_{n \in \mathbb{N}} \subset \mathbb{R}^I$ with $\overrightarrow  \Phi^n \to \overrightarrow  \Phi^\delta$ in $\mathbb{R}^I$, the sequence of corresponding regularized minimizers is $\|\cdot\|_B$ bounded for $\|\cdot\|_B$ as in (\ref{eq:permeability_norm_B}).
\item[(iii)] (Convergence) Assume additionally that the choice of $\alpha$ satisfies
\begin{equation*}
\alpha(\delta, \overrightarrow  \Phi^\delta) \to 0 \quad \text{\; and \;} \quad \frac{\delta^2}{\alpha(\delta, \overrightarrow  \Phi^\delta)} \le c_0, \qquad \text{\; as \;} \delta \to 0.
\end{equation*}
Then as $\delta \to 0$, $\overrightarrow  \Phi^\delta \to \overrightarrow  \Phi$, the family of minimizers $\left( \mu^\delta_{\alpha(\delta, \overrightarrow  \Phi^\delta)}, \overrightarrow{\hat A}^\delta_{\alpha(\delta, \overrightarrow  \Phi^\delta)}, \overrightarrow{\hat \psi}^\delta_{\alpha(\delta, \overrightarrow  \Phi^\delta)} \right)$ converges $\mathcal{T}$ subsequentially to a solution $(\mu^\dagger, \overrightarrow{\hat A}^\dagger,\overrightarrow{\hat \psi}^\dagger)$ to the problem (\ref{eq:permeability_problem_mixed}) with exact data $\overrightarrow  \Phi$. 
\end{itemize}
\end{corollary}

\begin{proof}
The proof is very similar to the proof of Corollary \ref{cor:CEM-EIT}. Here we only verify the $\mathcal{T}$ lower semicontinuity of $\mathcal{Q}_E$ in Assumption \ref{ass:Maao}(iv) and Assumption \ref{ass:Maao_newcondition}.
\begin{itemize}
\item The $\mathcal{T}$ lower semicontinuity of $\mathcal{Q}_E$ is verified by the continuity of $Q_i$ according to \eqref{eq:permeability_setting_Q_i}.
Indeed, we have
\begin{equation*}
\begin{split}
& |Q_i (\mu^n, \hat A^n_i, \hat \psi^n_i; \overrightarrow  \Phi) - Q_i (\mu, \hat A_i, \hat \psi_i; \overrightarrow  \Phi)| \\
& \; = \Bigg| \frac{1}{2} \underbrace{\int_\Omega \left( \mu^n_i |\nabla \hat \psi^n_i + A^J_i|^2 - \mu_i |\nabla \hat \psi_i + A_i^J|^2 \right)\dd x}_{I_1} \\
& \quad + \frac{1}{2} \underbrace{\int_\Omega \left( \frac{1}{\mu^n_i} |\nabla \times (\hat A^n_i + C^\ri_{A_i} (\overrightarrow  \Phi))|^2 - \frac{1}{\mu_i} |\nabla \times (\hat A_i + C^\ri_{A_i} (\overrightarrow  \Phi))|^2 \right) \dd x}_{I_2} \\
& \quad - \underbrace{\int_\Omega \left( (\nabla \hat \psi^n_i + A^J_i) \cdot \nabla \times (\hat A^n_i + C^\ri_{A_i} (\overrightarrow  \Phi)) - (\nabla \hat \psi_i + A^J_i) \cdot \nabla \times (\hat A_i + C^\ri_{A_i} (\overrightarrow  \Phi)) \right) \dd x}_{I_3} \Bigg|,
\end{split}
\end{equation*}
where
\begin{equation*}
\begin{split}
|I_1| & = \Big| \int_\Omega (\mu^n_i - \mu_i) |\nabla \hat \psi_i + A^J_i|^2 \dd x + \int_\Omega \mu^n_i \left( |\nabla \hat \psi^n_i + A^J_i|^2 - |\nabla \hat \psi_i + A^J_i|^2 \right) \dd x \Big| \\
& \le \Big| \int_\Omega (\mu^n_i - \mu_i) |\nabla \hat \psi_i + A^J_i|^2 \dd x \Big| \\
& \qquad \qquad + \overline{\mu} \Big| \int_\Omega ( \nabla \hat \psi^n_i + \nabla \hat \psi_i + 2 A^J_i) \cdot (\nabla \hat \psi^n_i - \nabla \hat \psi_i) \dd x \Big| \\
& \le \Big| \int_\Omega (\mu^n_i - \mu_i) |\nabla \hat \psi_i + A^J_i|^2 \dd x \Big| \\
& \qquad \qquad + \overline{\mu} \| \nabla \hat \psi^n_i + \nabla \hat \psi_i + 2 A^J_i \|_{L^2(\Omega)^3} \| \nabla \hat \psi^n_i - \nabla \hat \psi_i \|_{L^2(\Omega)^3} \\
& \to 0, \text{\; as \;} n \to \infty,
\end{split}
\end{equation*}
since $\mu^n_i \xrightharpoonup{*} \mu_i$ in $L^\infty (\Omega)$, $|\nabla \hat \psi_i + A^J_i|^2 \in L^1 (\Omega)$, and $\nabla \hat \psi^n_i \to \nabla \hat \psi_i$ in $L^2(\Omega)^3$; $I_2 \to 0$ in a similar way; and the absolute value of the last integral
\begin{equation*}
\begin{split}
|I_3| & \le \Bigg| \int_\Omega (\nabla \hat \psi^n_i - \nabla \hat \psi_i) \cdot \nabla \times (\hat A_i + C^\ri_{A_i} (\overrightarrow  \Phi)) \dd x \Bigg| \\
& \quad \qquad + \Bigg| \int_\Omega (\nabla \hat \psi^n_i + A^J_i) \cdot \left( \nabla \times (\hat A^n_i + C^\ri_{A_i} (\overrightarrow  \Phi)) - \nabla \times (\hat A_i + C^\ri_{A_i} (\overrightarrow  \Phi)) \right) \dd x \Bigg| \\
& \le \| \nabla \hat \psi^n_i - \nabla \hat \psi_i \|_{L^2(\Omega)^3} \|\nabla \times (\hat A_i + C^\ri_{A_i} (\overrightarrow  \Phi))\|_{L^2(\Omega)^3} \\
& \quad \qquad + \|\nabla \hat \psi^n_i + A^J_i\|_{L^2(\Omega)^3} \| \nabla \times \hat A^n_i - \nabla \times \hat A_i\|_{L^2(\Omega)^3} \\
& \to 0, \text{\; as \;} n \to \infty,
\end{split}
\end{equation*}
since $\nabla \hat \psi^n_i \to \nabla \hat \psi_i$ in $L^2(\Omega)^3$, and $\nabla \times \hat A^n_i \to \nabla \times \hat A_i$ in $L^2(\Omega)^3$.
\item Assumption \ref{ass:Maao_newcondition} is verified by Remark \ref{rem:Ass3} with $W=L^2(\Omega)^3$, 
\[
D_i(\mu, \hat A_i, \hat \psi_i) C^\ri(\overrightarrow  \Phi) = - \frac{1}{\sqrt{\mu_i}} \nabla \times C^\ri_{A_i} (\overrightarrow  \Phi)
= - \frac{1}{\sqrt{\mu_i}|\Gamma_c|} (0,0,1)^T \Phi_i
\]
cf. \eqref{eq:permeability_Cri_A}.
\end{itemize}
\end{proof}

\subsection{Localization of sound sources from microphone array measurements} \label{sec:SoundSources}
The problem of localizing sound sources from remote measurements of the sound pressure arises in a multitude of applications, such as failure diagnosis and monitoring as well as in sound design or noise reduction tasks. Under the simplifying assumption of unperturbed sound propagation in free space, it basically reduces to a signal processing problem (more precisely, deconvolution with respect to the free space Green's function for the Helmholtz equation) and can be solved by so-called beamforming methods and refined variants thereof, see, e.g., \cite{Mueller2002,Brooks2004,Sijtsma2009}. In realistic experimental scenarios, more complicated geometries, in particular bounded domains with combinations of reflecting and partially absorbing wall parts need to be taken into account by considering the wave or Helmholtz equation with appropriate boundary conditions as a model, cf., e.g., \cite{KKG18,Schumacher03}.

\subsubsection{The minimization form of the problem}
Acoustic wave propagation in a linear low amplitude regime is governed by the well-known wave equation
\begin{equation}\label{wave}
\frac{1}{c_0^2}p_{tt}-\Delta p = \sigma
\end{equation}
where $c_0$ is the speed of sound, $p$ is the acoustic pressure and $\sigma=\sigma(x)$ represents the distributed sound sources. We here, like in the previous sections, aim at formulating the problem as a system of first order PDEs and therefore go back to the (linearized versions of the) fundamental physical laws governing acoustic wave propagation, 
\[
\left\{
\begin{split}
&\mbox{linearized conservation of momentum: }\ \varrho_0 v_t + \nabla p_\sim = f, \\
&\mbox{linearized conservation of mass: }\ \varrho_{\sim t} + \varrho_0 \nabla \cdot v = g, \\
&\mbox{linearized equation of state: }\ \varrho_\sim = \tfrac{1}{c_0^2} p_\sim,
\end{split}
\right.
\]
which can be rephrased as 
\begin{equation} \label{eq:ss_linear_acoustic}
\left\{
\begin{split}
& \varrho_0 v_t + \nabla p_\sim = f, \\
& \tfrac{1}{c_0^2} p_{\sim t} + \varrho_0 \nabla \cdot v = g,
\end{split}
\right.
\end{equation}
where (with the subscripts $_\sim$ and $_0$ denoting fluctuating part and constant mean value, respectively)
\begin{itemize}
\item $\varrho = \varrho_0 + \varrho_\sim$ is the mass density 
\item $v$ is the acoustic particle velocity,
\item $p = p_0 + p_\sim$ is the pressure,
\item $c_0$ is the speed of sound.
\end{itemize}
Note that the second order wave equation \eqref{wave} can be derived from this by subtracting the divergence of the first line from the time derivative of the second line, thus eliminating the velocity. Via the identity $\sigma = g_t-\nabla\cdot f$, the functions $g$ and $f$ represent the searched for sound sources. As a matter of fact, in the identification process below, one may skip one of the two functions $f$ or $g$ -- preferably the latter, since $\sigma= -\nabla\cdot f$ allows to represent nonsmooth sources while still dealing with an $L^p$ function $f$ and also since it appears physically more meaningful to regard a sound source as giving rise to a momentum rather than giving rise to a change of mass density .

Imposing the above equations \eqref{eq:ss_linear_acoustic} on a domain $\Omega\subseteq\mathbb{R}^d$ with $d\in\{2,3\}$ and taking the Fourier transform with respect to time $t$ we get
\begin{equation} \label{eq:ss_linear_acoustic_ft}
\left\{
\begin{split}
\varrho_0 \I \omega v^\ft + \nabla p^\ft = f^\ft & \text{\; in \;} \Omega, \\
\tfrac{1}{c_0^2} \I \omega p^\ft + \varrho_0 \nabla \cdot v^\ft = g^\ft & \text{\; in \;} \Omega,
\end{split}
\right.
\end{equation}
with $v^\ft := \mathcal{F}^t v$, $p^\ft := \mathcal{F}^t p_\sim$, $f^\ft := \mathcal{F}^t f$, $g^\ft := \mathcal{F}^t g$, $\Omega \subset \mathbb{R}^d$ $(d \ge 2)$. The boundary of $\Omega$ is assumed to consist two parts, $\partial \Omega = \Gamma_r \cup \Gamma_a$, where $\Gamma_r$ is the sound hard part of the boundary  and $\Gamma_a = \partial \Omega \backslash \Gamma_r$ is the set of absorbing walls, see \cite{PTTW18} or represents a nonreflecting boundary that is used to enable truncation of the computational domain, see, e.g., the classical reference \cite{EngquistMajda1977} and the citing literature. Correspondingly, we impose the boundary conditions 
\begin{equation} \label{eq:ss_linear_acoustic_boundaryCondition}
\left\{
\begin{split}
\varrho_0 v^\ft \cdot \nu + \kappa p^\ft = 0 & \text{\; on \;} \Gamma_a, \\
v^\ft \cdot \nu = 0 & \text{\; on \;} \Gamma_r,
\end{split}
\right.
\end{equation}
where $\kappa \in \mathbb{R}$ is a positive constant depending on the properties of walls on $\Gamma_a$; typically $\kappa=c$ in case of computational absorbing boundary conditions (ABC).

By separating the real and imaginary parts of (\ref{eq:ss_linear_acoustic_ft}) with
\begin{equation*}
f^\ft = f_\Re + \I f_\Im, \qquad g^\ft = g_\Re + \I g_\Im, \qquad p^\ft = p_\Re + \I p_\Im, \qquad v^\ft = v_\Re + \I v_\Im,
\end{equation*}
we can see that the model for this problem is $E \left( f_\Re, f_\Im,  p_\Re, p_\Im, v_\Re, v_\Im \right) = 0$ with
\begin{equation} \label{eq:ss_model}
E \left(  f_\Re, f_\Im, g_\Re, g_\Im, p_\Re, p_\Im, v_\Re, v_\Im \right) =
\begin{pmatrix}
- \varrho_0 \omega v_\Im + \nabla p_\Re - f_\Re \\
\varrho_0 \omega v_\Re + \nabla p_\Im - f_\Im \\
- \frac{1}{c_0^2} \omega p_\Im + \varrho_0 \nabla \cdot v_\Re - g_\Re \\
\frac{1}{c_0^2} \omega p_\Re + \varrho_0 \nabla \cdot v_\Im - g_\Im
\end{pmatrix}.
\end{equation}
This allows to work exclusively in spaces of real valued functions and thus to avoid the potential trouble arising from nondifferentiability of the squared absolute value function $z\mapsto|z|^2$ in $\mathbb{C}$.

Taking again the $L^2(\Omega)$ norm of the model residual for defining the cost function (see \eqref{eq:ss_QP}, \eqref{Q1Q2Q3Q4} below), suggests to use the function space setting $(p_\Re, p_\Im) \in H^1(\Omega)^2$ and $(v_\Re, v_\Im) \in H(\dive, \Omega)^2$. Combining with the boundary conditions (\ref{eq:ss_linear_acoustic_boundaryCondition}) and noting that $v_\Re \cdot \nu, v_\Im \cdot \nu \in H^{-1/2} (\partial \Omega)$ (see \cite[Theorem 3.24]{Mon03}), we get the definition of the space $V$ containing the state $(p_\Re, p_\Im, v_\Re, v_\Im)$ as
\begin{equation} \label{eq:ss_V0}
\begin{split}
V &\subseteq V_0:=\Big\{ (p_\Re, p_\Im, v_\Re, v_\Im) \in H^1 (\Omega)^2 \times H (\dive,\Omega)^2: \\
& \qquad \qquad
\varrho_0 v_\Re \cdot \nu + \kappa p_\Re = 0, \quad \varrho_0 v_\Im \cdot \nu + \kappa p_\Im = 0 \text{\; in \;} H^{-1/2}(\Gamma_a), \\
& \qquad \qquad v_\Re \cdot \nu = 0, \quad v_\Im \cdot \nu = 0  \text{\; in \;} H^{-1/2} (\Gamma_r)
\Big\}.
\end{split}
\end{equation}
Microphone array measurements of the acoustic pressure are modelled by point values $p(x_\ell)$, $\ell\in\{1,\ldots,L\}$, where $x_\ell\in\Omega$ denotes the (known) location of the $\ell$-th microphone. 
Thus, the inverse problem under consideration is to find $(f_\Re, f_\Im, g_\Re, g_\Im) \in X$ with
\begin{equation} \label{eq:ss_X0}
X \subseteq X_0:=L^2 (\Omega)^d \times L^2 (\Omega)^d \times L^2 (\Omega) \times L^2 (\Omega),
\end{equation}
from the data 
\begin{equation} \label{eq:ss_C}
y=C(p_\Re, p_\Im, v_\Re, v_\Im):= (p_\Re (x_\ell),p_\Im (x_\ell))_{\ell=1}^L \in \mathbb{R}^{2L}.
\end{equation}
To guarantee sufficient regularity of $\hat p_\Re, \hat p_\Im$ at the measurement points so that the observation operator $C:V\to \mathbb{R}^{2L}$ is bounded -- note that $H^1$ functions in general do not admit point evaluation -- we assume that the support of the sources is separated from the measurement domain, i.e., 
\begin{eqnarray}
&&\Omega_{ms}\subseteq\Omega\,,\ \Omega_{ms}\mbox{ open } \label{eq:ss_X}\\
&& X=\{(f_\Re, f_\Im, g_\Re, g_\Im)\in X_0\ : \ \mbox{suppess}(h)\,\subseteq\Omega\setminus\Omega_{ms}\,, \ 
h\in\{f_\Re,f_\Im,g_\Re,g_\Im\}\,\}\,, 
\nonumber\\
&& V=\{(p_\Re, p_\Im, v_\Re, v_\Im)\in V_0\ : \ p_\Re\vert_{\Omega_{ms}},\, p_\Im\vert_{\Omega_{ms}}
\,\in C(\Omega_{ms})\}\,.
\label{eq:ss_V}
\end{eqnarray}
Indeed, for Fourier transformed solutions $p$ of \eqref{wave}, $H^2$ smoothness, and therefore -- via Sobolev's embedding -- continuity, follows immediately from interior regularity results for the homogeneous Helmholtz equation on $\Omega_{ms}$; therefore the exact solution of the inverse problem is indeed contained in $V$.

To solve the problem using the theory from Section \ref{sec:Preliminaries}, we choose a right inverse $C^\ri$ of $C$ as 
\begin{equation} \label{eq:ss_Cri}
C^\ri: \mathbb{R}^{2L} \to V, \qquad C^\ri (y) = \left( \sum_{\ell=1}^L y_{\Re,\ell} p_{0\Re,\ell}, \sum_{\ell=1}^L y_{\Im,\ell} p_{0\Im,\ell}, 0, 0 \right),
\end{equation}
where $y=(y_{\Re,1}, y_{\Im,1}, \dots, y_{\Re,L}, y_{\Im,L})\in\mathbb{R}$ and the functions $p_{0\Re,\ell} \in H^1 (\Omega)$, $p_{0\Im,\ell} \in H^1 (\Omega)$ are chosen such that $(p_{0\Re,\ell}, p_{0\Im,\ell}, 0, 0) \in V$,  
and $p_{0\Re,\ell} (x_j) = p_{0\Im,\ell} (x_j) = \delta_{\ell j}$ for $\ell,j\in\{1,\ldots,L\}$.

Now we rewrite our problem in the form of (\ref{eq:model_observation_mixed}) as
\begin{equation} \label{eq:ss_model_mixed}
\left\{
\begin{split}
- \varrho_0 \omega v_\Im + \nabla \left( \hat p_\Re + \sum_{\ell=1}^L y_{\Re,\ell} p_{0\Re,\ell} \right) - f_\Re = 0 \\
\varrho_0 \omega v_\Re + \nabla \left( \hat p_\Im + \sum_{\ell=1}^L y_{\Im,\ell} p_{0\Im,\ell} \right) - f_\Im = 0 \\
- \frac{1}{c_0^2} \omega \left( \hat p_\Im + \sum_{\ell=1}^L y_{\Im,\ell} p_{0\Im,\ell} \right) + \varrho_0 \nabla \cdot v_\Re - g_\Re = 0 \\
\frac{1}{c_0^2} \omega \left( \hat p_\Re + \sum_{\ell=1}^L y_{\Re,\ell} p_{0\Re,\ell} \right) + \varrho_0 \nabla \cdot v_\Im - g_\Im = 0 \\
(f_\Re, f_\Im, g_\Re, g_\Im, \hat p_\Re, \hat p_\Im, v_\Re, v_\Im) \in X \times \Ima (C),
\end{split}
\right.
\end{equation}
where $\Ima (C) = \{ (\hat p_\Re, \hat p_\Im, v_\Re, v_\Im) \in V: \hat p_{\Re} (x_\ell) = \hat p_{\Im} (x_\ell) = 0, \forall \ell\in \{1,\dots, I \} \}$. 

The cost function $\mathcal{J} = \mathcal{Q}_E$ is chosen as
\begin{equation} \label{eq:ss_QP}
\begin{split}
& \mathcal{Q}_E \left(  f_\Re, f_\Im, g_\Re, g_\Im, \hat p_\Re, \hat p_\Im, v_\Re, v_\Im; y \right) \\
& \qquad = \mathcal{Q}_1 (f_\Re, \hat p_\Re, v_\Im, y_\Re) + \mathcal{Q}_2 (f_\Im, \hat p_\Im, v_\Re, y_\Im) \\
& \qquad + \mathcal{Q}_3 (g_\Re, \hat p_\Im, v_\Re, y_\Im) + \mathcal{Q}_4 (g_\Im, \hat p_\Re, v_\Im, y_\Re)
\end{split}
\end{equation}
where $y_\Re = (y_{\Re,\ell})_{\ell=1}^L$, $y_\Im = (y_{\Im,\ell}))_{\ell=1}^L$ and
\begin{equation}\label{Q1Q2Q3Q4}
\begin{split}
&\mathcal{Q}_{1/2} (f_{\Re/\Im}, \hat p_{\Re/\Im}, v_{\Im/\Re}, y_{\Re/\Im})  
:=\begin{cases}\frac12\|q_{1/2}\|_{L^2(\Omega)^d}^2 \mbox{ if } q_{1/2}\in L^2(\Omega)^d\\
+\infty\mbox{ else}
\end{cases}
\\
&\mathcal{Q}_{3/4} (g_{\Re/\Im}, \hat p_{\Im/\Re}, v_{\Re/\Im}, y_{\Im/\Re})  
:=\begin{cases}\frac12\|q_{3/4}\|_{L^2(\Omega)}^2 \mbox{ if } q_{3/4}\in L^2(\Omega)\\
+\infty\mbox{ else}
\end{cases}
\\
\mbox{with }
&q_1:= - \varrho_0 \omega v_\Im + \nabla \left( \hat p_\Re + \sum_{\ell=1}^L y_{\Re,\ell} p_{0\Re,\ell} \right) - f_\Re\\
&q_2:=\varrho_0 \omega v_\Re + \nabla \left( \hat p_\Im + \sum_{\ell=1}^L y_{\Im,\ell} p_{0\Im,\ell} \right) - f_\Im\\
&q_3:= - \frac{1}{c_0^2} \omega \left( \hat p_\Im + \sum_{\ell=1}^L y_{\Im,\ell} p_{0\Im,\ell} \right) + \varrho_0 \nabla \cdot v_\Re - g_\Re\\
&q_4:= \frac{1}{c_0^2} \omega \left( \hat p_\Re + \sum_{\ell=1}^L y_{\Re,\ell} p_{0\Re,\ell} \right) + \varrho_0 \nabla \cdot v_\Im - g_\Im ,
\end{split}
\end{equation}
So differently from the previous two examples, the number of measurements is not reflected in the number of terms in the cost functional. The measurements are  rather imposed via the requirement of $\hat{p}_\Re$ and $\hat{p}_\Im$ vanishing (or being smaller than $\tau\delta$) at the measurement points $x_\ell$.

The choices for the regularization functional $\mathcal{R}$ and the discrepancy $\mathcal{S}$ are
\begin{equation} \label{eq:ss_S}
\mathcal{S} (y,z) = \| z - y \|_{\infty,\mathbb{R}^{2L}},
\end{equation}
\begin{equation} \label{eq:ss_R}
\begin{split}
& \mathcal{R} \left( f_\Re, f_\Im, g_\Re, g_\Im, \hat p_\Re, \hat p_\Im, v_\Re, v_\Im \right) \\
& \qquad = \frac{1}{2} \| f_\Re \|^2_{L^2(\Omega)^d} + \frac{1}{2} \| f_\Im \|^2_{L^2(\Omega)^d} + \frac{1}{2} \| g_\Re \|^2_{L^2(\Omega)} + \frac{1}{2} \| g_\Im \|^2_{L^2(\Omega)} \\
& \qquad \qquad + \frac{1}{2} \| \hat p_\Re \|^2_{L^2(\Omega)} + \frac{1}{2} \| \hat p_\Im \|^2_{L^2(\Omega)} + \frac{1}{2} \| v_\Re \|^2_{L^2(\Omega)^d} + \frac{1}{2} \| v_\Im \|^2_{L^2(\Omega)^d},
\end{split}
\end{equation}
As a matter of fact, boundedness of the $L^2$ norms of  $\hat p_\Re$, $\hat p_\Im$, $v_\Re$, $v_\Im$, combined with boundedness of $\mathcal{Q}_E$ will allow us to bound higher order norms of these states.
Also note that here the problem of needing the regularization term in order to guarantee existence of a minimizer does not arise here, as opposed to the two previous examples. This is essentially due to the linearity of the inverse problem under consideration here.

Quite often in practice, sources are supported on a set of points or along lines in two or three dimensional space. In order to account for this we enhance sparsity of the recovered sources by using the functional
\begin{equation} \label{eq:ss_Rtilde}
\begin{split}
&\widetilde{\mathcal{R}} (f_\Re, f_\Im, g_\Re, g_\Im, \hat p_\Re, \hat p_\Im, v_\Re, v_\Im) \\
&= \| \nabla \cdot f_\Re \|_{\mathcal{M}} + \| \nabla \cdot f_\Im \|_{\mathcal{M}} + \| \nabla g_\Re \|_{\mathcal{M}} + \| \nabla g_\Im \|_{\mathcal{M}},
\end{split}\end{equation}
where $\mathcal{M} = C_b(\Omega)^*$ is the space of Radon measures on $\Omega$, cf. \cite{BrediesPikkarainen13,CasasClasonKunisch12,PTTW18}.

Again, instead of the exact data $y=(y_{\Re,1}, y_{\Im,1}, \dots, y_{\Re,L}, y_{\Im,L})$, we only have a noisy version $y^\delta = (y^\delta_{\Re,1}, y^\delta_{\Im,1}, \dots, y^\delta_{\Re,L}, y^\delta_{\Im,L})$ with the accuracy of $\delta$ in the sense of
$$\mathcal{S} (y, y^\delta) = \| y^\delta - y \|_{\infty,\mathbb{R}^{2L}} \le \delta$$
and the regularized problem is to find
\begin{equation} \label{eq:ss_model_mixed_noisy}
\begin{split}
\argmin & \Big\{ \mathcal{Q}_E \left(  f_\Re, f_\Im, g_\Re, g_\Im, \hat p_\Re, \hat p_\Im, v_\Re, v_\Im; y^\delta \right) + \alpha \mathcal{R} \left(  f_\Re, f_\Im, g_\Re, g_\Im, \hat p_\Re, \hat p_\Im, v_\Re, v_\Im \right) \Big \},\\
& \text{\; s.t. \;} \left\{
\begin{array} {l} \left(  f_\Re, f_\Im, g_\Re, g_\Im, \hat p_\Re, \hat p_\Im, v_\Re, v_\Im \right) \in X \times V, \\
\max_{\ell\in \{1, \dots, L\}} \{ |\hat p_\Re (x_\ell)|, |\hat p_\Im (x_\ell)| \} \le \tau \delta, \\
\widetilde{\mathcal{R}} (f_\Re, f_\Im, g_\Re, g_\Im, \hat p_\Re, \hat p_\Im, v_\Re, v_\Im) \le \rho,
\end{array} \right.
\end{split}
\end{equation}
where $\rho>0$ is a given constant (note to be mistaken with the mass density here) such that the exact solution satisfies 
\begin{equation} \label{eq:ss_priori_information_fg}
\| \nabla \cdot f_\Re^\dagger \|_{\mathcal{M}} + \| \nabla \cdot f_\Im^\dagger \|_{\mathcal{M}} + \| \nabla g_\Re^\dagger \|_{\mathcal{M}} + \| \nabla g_\Im^\dagger \|_{\mathcal{M}} \le \rho
\end{equation}
and $\tau > 1$ is fixed.

\subsubsection{Convergence}
To achieve convergence results, we use the priori information \eqref{eq:ss_priori_information_fg}
and choose the topology $\mathcal{T}$ and the norm $\| \cdot \|_B$ as 
\begin{equation} \label{eq:ss_topologyT}
\begin{split}
& (f_\Re^n, f_\Im^n, g_\Re^n, g_\Im^n, \hat p_\Re^n, \hat p_\Im^n, v_\Re^n, v_\Im^n) \xrightarrow{\mathcal{T}} (f_\Re, f_\Im, g_\Re, g_\Im, \hat p_\Re, \hat p_\Im, v_\Re, v_\Im) \\
& \qquad \Leftrightarrow \left\{
\begin{array} {l}
f_\Re^n \rightharpoonup f_\Re, \quad f_\Im^n \rightharpoonup f_\Im \quad \text{\; in \;} L^2(\Omega)^d, \\
g_\Re^n \rightharpoonup g_\Re, \quad g_\Im^n \rightharpoonup g_\Im \quad \text{\; in \;} L^2(\Omega), \\
\hat p_\Re^n \rightharpoonup \hat p_\Re, \quad \hat p_\Im^n \rightharpoonup \hat p_\Im \quad \text{\; in \;} H^1(\Omega), \\
v_\Re^n \rightharpoonup v_\Re, \quad v_\Im^n \rightharpoonup v_\Im \quad \text{\; in \;} H(\dive, \Omega), \\
\nabla \cdot f_\Re^n \xrightharpoonup{*} \nabla \cdot f_\Re, \quad \nabla \cdot f_\Im^n \xrightharpoonup{*} \nabla \cdot f_\Im \quad \text{\; in \;} \mathcal{M}, \\
\nabla g_\Re^n \xrightharpoonup{*} \nabla g_\Re, \quad \nabla g_\Im^n \xrightharpoonup{*} \nabla g_\Im \quad \text{\; in \;} \mathcal{M}, \\
\left( \hat p_\Re^n (x_\ell), \hat p_\Im^n (x_\ell) \right)_{\ell=1}^L \to \left( \hat p_\Re (x_\ell), \hat p_\Im (x_\ell) \right)_{\ell=1}^L\quad \text{\; in \;} \mathbb{R}^{2L},
\end{array}
\right.
\end{split}
\end{equation}

\begin{equation} \label{eq:ss_normB}
\begin{split}
& \| (f_\Re, f_\Im, g_\Re, g_\Im, \hat p_\Re, \hat p_\Im, v_\Re, v_\Im) \|_B \\
& \qquad = \| f_\Re \|_{L^2(\Omega)^d} + \| f_\Im \|_{L^2(\Omega)^d} + \| g_\Re \|_{L^2(\Omega)} + \| g_\Im \|_{L^2(\Omega)} \\
& \qquad \quad + \| \nabla \cdot f_\Re \|_{\mathcal{M}} + \| \nabla \cdot f_\Im \|_{\mathcal{M}} + \| \nabla g_\Re \|_{\mathcal{M}} + \| \nabla g_\Im \|_{\mathcal{M}} \\
& \qquad \quad + \| \hat p_\Re \|_{H^1(\Omega)} + \| \hat p_\Im \|_{H^1(\Omega)} + \| v_\Re \|_{H(\dive, \Omega)} + \| v_\Im \|_{H(\dive,\Omega)} \\
& \qquad \quad + \| \left( \hat p_\Re (x_\ell), \hat p_\Im (x_\ell) \right)_{\ell=1}^L \|_{\mathbb{R}^{2L}}.
\end{split}
\end{equation}

\begin{corollary} \label{cor:ss}
Let (\ref{eq:ss_priori_information_fg}) hold.
\begin{itemize}
\item[(i)] (Existence of minimizers) Then for any $\alpha>0$, a minimizer of (\ref{eq:ss_model_mixed_noisy}) exists;
\item[(ii)] (Boundedness) and for any sequence $(y^n)_{n \in \mathbb{N}} \subset \mathbb{R}^{2L}$ with $y^n \to y^\delta$ in $\mathbb{R}^{2L}$, the sequence of corresponding regularized minimizers is $\|\cdot\|_B$ bounded.
\item[(iii)] (Convergence) Assume additionally that the choice of $\alpha$ satisfies
\begin{equation*}
\alpha(\delta, y^\delta) \to 0 \quad \text{\; and \;} \quad \frac{\delta^2}{\alpha(\delta, y^\delta)} \le c_0, \qquad \text{\; as \;} \delta \to 0.
\end{equation*}
Then as $\delta \to 0$, $y^\delta \to y$, the family of minimizers of (\ref{eq:ss_model_mixed_noisy}) converges $\mathcal{T}$ subsequentially to a solution to the problem (\ref{eq:ss_model_mixed}) with the exact data $y$. 
\end{itemize}
\end{corollary}

\begin{proof}
We will check each item in the Assumptions \ref{ass:Maao} and \ref{ass:Maao_newcondition} as below.
\begin{itemize}
\item Assumption \ref{ass:Maao} (i) follows from (\ref{eq:ss_priori_information_fg}) and the Assumption \ref{ass:Maao} (ii) from the definition of $\mathcal{R}$ (\ref{eq:ss_R}).
\item The $\|\cdot\|_B$ boundedness of $L_c$ follows from the definitions of $\mathcal{Q}_E$, $\mathcal{R}$, $\widetilde{\mathcal{R}}$, $\mathcal{S}$. The $\mathcal{T}$ compactness of $L_c$ follows from weak compactness of bounded sets in Hilbert spaces and weak* compactness of bounded subsets of $\mathcal{M}$ since it is the dual of a separable space. Hence, Assumption \ref{ass:Maao} (iii) is verified.
\\
Note that here in order to bound the $H^1 (\Omega)^2 \times H (\dive,\Omega)^2$ norm of the state we use the triangle inequality together with boundedness of $\mathcal{Q}_E$ and $\mathcal{R}$, e.g.
\[
\begin{split}
\|\nabla \hat{p}_\Re\|_{L^2(\Omega)^d}\leq&
\sqrt{2Q_1(f_\Re, \hat p_\Re, v_\Im, z_\Re)}
+\varrho_0\omega \|v_\Im\|_{L^2(\Omega)^d}\\
&+\|z_\Re\|_{\infty,\mathbb{R}^{2L}}\sum_{\ell=1}^L\|\nabla p_{0\Re,\ell}\|_{L^2(\Omega)^d}
+\|f_\Re\|_{L^2(\Omega)^d}\,.
\end{split}
\]
Doing so without including the $L^2$ norms of the states into $\mathcal{R}$, so just by using boundedness of $\mathcal{Q}_E$, of the $L^2$ norms of the sources, and continuity of the embeddings $H^1 (\Omega)\to L^2(\Omega)$ and $H (\dive,\Omega)\to L^2(\Omega)^d$ together with some elimination strategy would not work, since the zero derivative terms of the states in $\mathcal{Q}_E$ come with frequency dependent factors that will typically be larger than the embedding constants of $H^1(\Omega)\to L^2(\Omega)$ and $H (\dive,\Omega)\to L^2(\Omega)^d$.
\item Assumption \ref{ass:Maao} (iv) follows from linearity of the operators inside the norms and weak lower semicontinuity of the norms.
\item Assumption \ref{ass:Maao} (v) is obvious by Remark \ref{rem:ass:Maao(v)}.
\item Assumption \ref{ass:Maao_newcondition} is verified by Remark \ref{rem:Ass3} with $W=L^2(\Omega)^d$, 
\[
\begin{split}
&D_1(f_\Re, \hat p_\Re, v_\Im) C^\ri(z)=\sum_{\ell=1}^L z_{\Re,\ell}\nabla p_{0\Re,\ell}, \
D_3(g_\Re, \hat p_\Im, v_\Re) C^\ri(z)=-\frac{\omega}{c_0^2}\sum_{\ell=1}^L z_{\Im,\ell}\nabla p_{0\Im,\ell}\,,\\
&D_2(f_\Im, \hat p_\Im, v_\Re) C^\ri(z)=\sum_{\ell=1}^L z_{\Im,\ell}\nabla p_{0\Im,\ell}, \ 
D_4(g_\Im, \hat p_\Re, v_\Im) C^\ri(z)=-\frac{\omega}{c_0^2}\sum_{\ell=1}^L z_{\Re,\ell}\nabla p_{0\Im,\ell}\,,
\end{split}
\]
cf. \eqref{eq:ss_Cri}, \eqref{eq:ss_QP}, \eqref{Q1Q2Q3Q4}.
\end{itemize}
\end{proof}

\section{Conclusions and outlook}
In this paper we have contributed some further examples of variational formulations for inverse problems cf. \cite{AndrieuxBarangerBenAbda2006,BrownJaisKnowles2005,BrownJais2011,Knowles1998,KnowlesWallace1996,KohnVogelius87,KM90,Ron07} 
and additionally provided a regularization framework for these formulations, following up on \cite{Kal18} and augmenting it by the idea of data inversion using a right inverse $C^\ri$ of the observation operator.

Future research in this context will be on the question of optimally selecting this right inverse operator, as well as on further applications. While the examples in this paper are from electromagnetics and acoustics and there are certainly many other relevant inverse problems in this physical context, we are also interested in extending the scope to problems from elasticity as arising in medical (e.g., elastography) and engineering (e.g., nondestructive testing) applications.
Also time dependent problems will be considered in a next step -- the time domain version of the problem from Section \ref{sec:SoundSources} is already one of them. Other wave models (electromagnetics, elasticity) also allow for such a first order system formulation \cite{KirschRieder2016} and have many important real world applications.

\section*{Appendix}
In this section we will state a few relations between boundedness of the sequence $(\mu^n)_{n\in\mathbb{N}}$ in $L^\infty([0,\infty))$ and boundedness of $(\mu^n(h^n))_{n\in\mathbb{N}}$ in $L^\infty(\Omega)$ for $h^n,h\in C^{k,\beta}(\Omega)$ $h^n\to h$ in $C^{\ell,\gamma}(\Omega)$.

First of all, for Lipschitz continuous functions $h^n:\Omega\to\mathbb{R}$ the inequality
\begin{equation}\label{munHnLinfty}
\|\mu^n(h^n)\|_{L^\infty(\Omega)} \geq \|\mu^n\|_{L^\infty(h^n(\Omega))},
\end{equation}
holds, which can be seen as follows.
Using the definition of the $L^\infty$ norm on $\Omega$ with the Borel sigma algebra and the Lebesgue measure $\lambda^d$
\[
\|\mu^n(h^n)\|_{L^\infty(\Omega)} = \inf R^n 
\]
where 
\begin{eqnarray*}
R^n&=&\{c\geq0\, : \, \exists N\subseteq\Omega, \,\lambda^d(N)=0\, \forall x\in \Omega\setminus N\, : \mu^n(h^n(x))\leq c\}\\
&=&\{c\geq0\, : \, \exists N\subseteq\Omega, \,\lambda^d(N)=0\, \forall z\in h^n(\Omega\setminus N)\, : \mu^n(z)\leq c\}
\end{eqnarray*}
and 
\[
h^n(\Omega)\setminus h^n(N)\subset h^n(\Omega\setminus N)\subset h^n(\Omega)\setminus h^n(\emptyset)
\]
we see that 
\[
c\in R^n\ \Rightarrow \ \exists \tilde{N}\subseteq h^n(\Omega), \,\lambda^1(\tilde{N})=0 \, \forall z\in h^n(\Omega)\setminus \tilde{N}\, : \mu^n(z)\leq c\,.
\]
Here we have made use of the fact that a Lipschitz continuous function maps sets of measure zero into sets of measure zero. Note that this is in general not true for H\"older continuous functions, the Cantor function being a well-known counterexample.
This proves \eqref{munHnLinfty}.

Now if $h^n$ converges to $h$ in $C(\Omega)$, then 
\[
\forall V\subset h(\Omega), \,\overline{V}\subset h(\Omega)^o\, \exists n_V\in\mathbb{N}\, \forall n\geq n_V: V\subseteq h^n(\Omega)
\]

Thus, altogether setting $C_V:= C+\max\{\|\mu^j\|_{L^\infty(V)}\, : \,j\in\{1,\ldots,n_V-1\}\}$ we have proven the following relation.
\begin{lemma}
Let, for all $n\in\mathbb{N}$, $\mu^n\in L^\infty(h(\Omega))$ and $h^n:\Omega\to\mathbb{R}$ be Lipschitz continuous and converge to $h$ in $C(\Omega)$, and assume that there exists $C>0$ such that 
\[
\forall n\in\mathbb{N} \, : \|\mu^n(h^n)\|_{L^\infty(\Omega)}\leq C
\]
Then for any $V\subset h(\Omega), \,\overline{V}\subset h(\Omega)^o$ there exists $C_V>0$ such that 
\[
\forall n\in\mathbb{N} \, : \|\mu^n\|_{L^\infty(V)}\leq C_V\,.
\]
\end{lemma}

Lipschitz continuity of $h^n$ and its convergence in $C(\Omega)$ along a subsequence can be achieved by choosing $s$ in the definition \eqref{eq:permeability_setting_R} of the regularization function $\mathcal{R}$ sufficiently large ($s>\frac{d}{2}+1$) and using Sobolev's embedding.
Note however, that this is not required for obtaining the well-definedness, boundedness and convergence results of Corollary \ref{cor:permeability}.

\medskip

On the other hand, $L^\infty([0,\infty))$ boundedness of the sequence $(\mu^n)_{n\in\mathbb{N}}$ clearly implies $L^\infty(\Omega)$ boundedness of $(\mu^n(h^n))_{n\in\mathbb{N}}$.
However the latter cannot be concluded from $L^\infty(h(\Omega))$ boundedness of $(\mu^n)_{n\in\mathbb{N}}$, even if $h^n\to h$ in $C^\infty(\Omega)$, as the simple counterexample $\Omega=(0,1)$, $h(x)=x$, $h^n(x)=x+\frac{1}{n}$, $\mu^n(z)=\begin{cases}0\mbox{ for }z\in(0,1)\\ \exp(3n-\frac{1}{z-1})\mbox{ for }z\geq1\end{cases}$ shows.

\section*{Acknowledgment}
This work was supported by the Austrian Science Fund FWF under the grant P30054.


\begin{thebibliography}{10}

\bibitem{AndrieuxBarangerBenAbda2006}
S~Andrieux, T~N Baranger, and A~BenAbda.
\newblock Solving cauchy problems by minimizing an energy-like functional.
\newblock {\em Inverse Problems}, 22(1):115, 2006.

\bibitem{Borcea2002}
Liliana Borcea.
\newblock Electrical impedance tomography.
\newblock {\em Inverse Problems}, 18(6):R99--R136, oct 2002.

\bibitem{BrediesPikkarainen13}
Kristian Bredies and Hanna~Katriina Pikkarainen.
\newblock Inverse problems in spaces of measures.
\newblock {\em ESAIM: Control, Optimisation and Calculus of Variations},
  19(1):190--218, 2013.

\bibitem{Brooks2004}
T.~F. Brooks and W.~M. Humphreys.
\newblock A deconvolution approach for the mapping of acoustic sources (damas)
  determined from phased microphone arrays.
\newblock {\em AIAA}, 2004.

\bibitem{BrownJais2011}
B~M Brown and M~Jais.
\newblock A variational approach to an electromagnetic inverse problem.
\newblock {\em Inverse Problems}, 27(4):045011, 2011.

\bibitem{BrownJaisKnowles2005}
B~M Brown, M~Jais, and I~W Knowles.
\newblock A variational approach to an elastic inverse problem.
\newblock {\em Inverse Problems}, 21(6):1953, 2005.

\bibitem{CasasClasonKunisch12}
Eduardo Casas, Christian Clason, and Karl Kunisch.
\newblock Approximation of elliptic control problems in measure spaces with
  sparse solutions.
\newblock {\em SIAM Journal on Control and Optimization}, 50(4):1735--1752,
  2012.

\bibitem{EHN96}
Heinz~Werner Engl, Martin Hanke, and A.~Neubauer.
\newblock {\em Regularization of Inverse Problems}.
\newblock Mathematics and Its Applications. Springer Netherlands, 1996.

\bibitem{EngquistMajda1977}
B.~Engquist and A.~Majda.
\newblock Absorbing boundary conditions for the numerical simulation of waves.
\newblock {\em Math. Comp.}, 31(139):629--651, 1977.

\bibitem{HKPS07}
Bernd Hofmann, Barbara Kaltenbacher, Christiane Pöschl, and O~Scherzer.
\newblock A convergence rates result for tikhonov regularization in banach
  spaces with non-smooth operators.
\newblock {\em Inverse Problems}, 23:987, 04 2007.

\bibitem{HW13}
Thorsten Hohage and Frank Werner.
\newblock Iteratively regularized newton-type methods for general data misfit
  functionals and applications to poisson data.
\newblock {\em Numerische Mathematik}, 123(4):745--779, Apr 2013.

\bibitem{comp_minIP}
Philipp Hungerl\"ander, Barbara Kaltenbacher, and Franz Rendl.
\newblock Regularization of inverse problems via box constrained minimization.
\newblock {\em Inverse Problems and Imaging}, 14:437--461, 2020.

\bibitem{IVT02}
Valentin~K. Ivanov, Vladimir~V. Vasin, and Vitalii~P. Tanana.
\newblock {\em Theory of linear ill-posed problems and its applications}.
\newblock Inverse and ill-posed problems series. Utrecht ; Boston, 2002.

\bibitem{JXZ16}
Bangti Jin, Yifeng Xu, and Jun Zou.
\newblock A convergent adaptive finite element method for electrical impedance
  tomography.
\newblock {\em IMA Journal of Numerical Analysis}, 37, 08 2016.

\bibitem{KKG18}
B.~Kaltenbacher, M.~Kaltenbacher, and S.~Gombots.
\newblock {Inverse scheme for acoustic source localization using microphone
  measurements and finite element simulations}.
\newblock {\em Acta Acustica united with Acustica}, 104:647--656, 2018.

\bibitem{Kal16}
Barbara Kaltenbacher.
\newblock Regularization based on all-at-once formulations for inverse
  problems.
\newblock {\em SIAM Journal on Numerical Analysis}, 54, 03 2016.

\bibitem{Kal18}
Barbara Kaltenbacher.
\newblock Minimization based formulations of inverse problems and their
  regularization.
\newblock {\em SIAM Journal on Optimization}, 28:620--645, 01 2018.

\bibitem{KKR03}
Barbara Kaltenbacher, Manfred Kaltenbacher, and Stefan Reitzinger.
\newblock Identification of nonlinear b–h curves based on magnetic field
  computations and multigrid methods for ill-posed problems.
\newblock {\em European Journal of Applied Mathematics}, 14:15 -- 38, 02 2003.

\bibitem{KKV14}
Barbara Kaltenbacher, Alana Kirchner, and Boris Vexler.
\newblock Goal oriented adaptivity in the irgnm for parameter identification in
  pdes: Ii. all-at-once formulations.
\newblock {\em Inverse Problems}, 30:045002, 02 2014.

\bibitem{KK18}
Barbara Kaltenbacher and Andrej Klassen.
\newblock On convergence and convergence rates for ivanov and morozov
  regularization and application to some parameter identification problems in
  elliptic {PDEs}.
\newblock {\em Inverse Problems}, 34(5):055008, apr 2018.

\bibitem{KRR16}
Barbara Kaltenbacher, Franz Rendl, and Elena Resmerita.
\newblock Computing quasisolutions of nonlinear inverse problems via efficient
  minimization of trust region problems.
\newblock {\em Journal of Inverse and Ill-posed Problems}, 24:435--447, 2016.

\bibitem{KirschRieder2016}
Andreas Kirsch and Andreas Rieder.
\newblock Inverse problems for abstract evolution equations with applications
  in electrodynamics and elasticity.
\newblock {\em Inverse Problems}, 32(8):085001, 2016.

\bibitem{Knowles1998}
Ian Knowles.
\newblock A variational algorithm for electrical impedance tomography.
\newblock {\em Inverse Problems}, 14(6):1513, 1998.

\bibitem{KnowlesWallace1996}
Ian Knowles and Robert Wallace.
\newblock A variational solution for the aquifer transmissivity problem.
\newblock {\em Inverse Problems}, 12(6):953, 1996.

\bibitem{KM90}
Robert~V. Kohn and Alan McKenney.
\newblock Numerical implementation of a variational method for electrical
  impedance tomography.
\newblock {\em Inverse Problems}, 6(3):389--414, Jun 1990.

\bibitem{KohnVogelius87}
Robert~V Kohn and Michael Vogelius.
\newblock Relaxation of a variational method for impedance computed tomography.
\newblock {\em Communications on Pure and Applied Mathematics}, 40(6):745--777,
  1987.

\bibitem{LW13}
Dirk Lorenz and Nadja Worliczek.
\newblock Necessary conditions for variational regularization schemes.
\newblock {\em Inverse Problems}, 29(7):075016, jun 2013.

\bibitem{Mon03}
Peter Monk.
\newblock {\em Finite Element Methods For Maxwell's Equations}.
\newblock Clarendon Press - Oxford, 2003.

\bibitem{Mueller2002}
T.~J. Mueller.
\newblock {\em Aeroacoustic Measurements}.
\newblock Springer-Verlag, 2002.

\bibitem{NR14}
Andreas Neubauer and Ronny Ramlau.
\newblock On convergence rates for quasi-solutions of ill-posed problems.
\newblock {\em ETNA. Electronic Transactions on Numerical Analysis}, 41:81--92,
  01 2014.

\bibitem{PTTW18}
Konstantin Pieper, Bao~Quoc Tang, Philip Trautmann, and Daniel Walter.
\newblock Inverse point source location with the helmholtz equation on a
  bounded domain.
\newblock 2018.
\newblock arXiv:1805.03310.

\bibitem{Ron07}
Luca Rondi.
\newblock A variational approach to the reconstruction of cracks by boundary
  measurements.
\newblock {\em Journal de Mathématiques Pures et Appliquées}, 87(3):324 --
  342, 2007.

\bibitem{SGG+09}
Otmar Scherzer, Markus Grasmair, Harald Grossauer, Markus Haltmeier, and Frank
  Lenzen.
\newblock {\em Variational Methods in Imaging}.
\newblock Applied Mathematical Sciences. Springer-Verlag New York, 2009.

\bibitem{Schumacher03}
A.~Schuhmacher, K.~Rasmussen, and C.~Hansen.
\newblock Sound source reconstruction using inverse boundary element
  calculations.
\newblock {\em J. Acoust. Soc. Am.}, 113:114--127, 2003.

\bibitem{Sijtsma2009}
P.~Sijtsma.
\newblock Clean based on spatial source coherence.
\newblock {\em Int. J. Aeroacoustics}, 6(6):357--374, 2009.

\bibitem{SCI92}
Erkki Somersalo, Margaret Cheney, and David Isaacson.
\newblock Existence and uniqueness for electrode models for electric current
  computed tomography.
\newblock {\em SIAM Journal on Applied Mathematics}, 52(4):1023--1040, 1992.

\end{thebibliography}
\end{document}